\documentclass[a4paper, reqno]{amsart}
\usepackage[dvipsnames]{xcolor}
\usepackage{amsmath,amsthm,amssymb,relsize,mathrsfs}
\usepackage[margin=30mm]{geometry}
\usepackage{tikz}
\usetikzlibrary{
  cd,
  calc,
  positioning,
  fit,
  arrows,
  decorations.pathreplacing,
  decorations.markings,
  shapes.geometric,
  backgrounds,
  bending
}
\usepackage{tikzsymbols}

\definecolor{darkblue}{rgb}{0,0,0.6}

\usepackage[breaklinks, pdftex, ocgcolorlinks,colorlinks=true, citecolor=darkblue, filecolor=darkblue, linkcolor=darkblue, urlcolor=darkblue]{hyperref}
\usepackage[capitalize,noabbrev]{cleveref}

\makeatletter
\newtheorem*{rep@theorem}{\rep@title}
\newcommand{\newreptheorem}[2]{%
	\newenvironment{rep#1}[1]{%
		\def\rep@title{#2 \ref{##1}}%
		\begin{rep@theorem}}%
		{\end{rep@theorem}}}
\makeatother

\newtheorem{proposition}{Proposition}[section]
\newtheorem{theorem}[proposition]{Theorem}
\newtheorem{corollary}[proposition]{Corollary}
\newtheorem{lemma}[proposition]{Lemma}
\newreptheorem{theorem}{Theorem}
\newreptheorem{proposition}{Proposition}

\theoremstyle{definition}
\newtheorem{definition}[proposition]{Definition}
\newtheorem{question}[proposition]{Question}

\theoremstyle{remark}
\newtheorem{remark}[proposition]{Remark}

\newtheorem*{remark*}{Remark}

\crefname{theorem}{Theorem}{Theorems}
\crefname{proposition}{Proposition}{Propositions}
\crefname{corollary}{Corollary}{Corollaries}
\crefname{definition}{Definition}{Definitions}
\crefname{lemma}{Lemma}{Lemmas}
\crefname{question}{Question}{Questions}
\crefname{example}{Example}{Examples}
\crefname{conjecture}{Conjecture}{Conjectures}
\crefname{remark}{Remark}{Remarks}
\crefname{const}{Construction}{Constructions}

\newcommand{\lk}{\operatorname{lk}}

\newcommand{\Q}{\mathbb{Q}}

\newcommand{\Z}{\mathbb{Z}}

\newcommand{\Id}{\operatorname{Id}}
\newcommand{\id}{\operatorname{Id}}
\newcommand{\Arf}{\operatorname{Arf}}

\newcommand{\ol}{\overline}

\newcommand{\wt}{\widetilde}

\newcommand{\pri}{\mathfrak{pri}}

\newcommand{\ks}{\operatorname{ks}}

\DeclareMathOperator{\Aut}{Aut}
\DeclareMathOperator{\Spin}{Spin}
\DeclareMathOperator{\GL}{GL}

\DeclareMathOperator{\SO}{SO}

\DeclareMathOperator{\BSO}{BSO}
\DeclareMathOperator{\BSpin}{BSpin}

\DeclareMathOperator{\Isom}{Isom}
\DeclareMathOperator{\simp}{simp}
\DeclareMathOperator{\Tors}{Tors}

\newcommand{\ie}{i.e.~}

\newcommand{\sumStwo}[1]{\ensuremath{\#_{#1} (S^2 \times S^2)}}

\begin{document}
\title{Algorithms in $4$-manifold topology}

\author[SB]{Stefan Bastl}
\author[RB]{Rhuaidi Burke}
\author[RC]{Rima Chatterjee}
\author[SD]{Subhankar Dey}
\author[AD]{Alison Durst}
\author[SF]{Stefan Friedl}
\author[DG]{Daniel Galvin}
\author[AG]{Alejandro Garc\'ia Rivas}
\author[TH]{Tobias Hirsch}
\author[CH]{Cara Hobohm}
\author[C-SH]{Chun-Sheng Hsueh}
\author[MK]{Marc Kegel}
\author[FK]{Frieda Kern}
\author[SL]{Shun Ming Samuel Lee}
\author[CL]{Clara L\"oh}
\author[NM]{Naageswaran Manikandan}
\author[LM]{L\'eo Mousseau}
\author[LM]{Lars Munser}
\author[MP]{Mark Pencovitch}
\author[PP]{Patrick Perras}
\author[MP]{Mark Powell}
\author[JPQ]{Jos\'e Pedro Quintanilha}
\author{Lisa Schambeck}
\author[DS]{David Suchodoll}
\author[MT]{Martin Tancer}
\author[AT]{Annika Thiele}
\author[PT]{Paula Tru\"ol}
\author[MU]{Matthias Uschold} 
\author[SV]{Simona Vesel\'a}
\author[MW]{Melvin Weiß}
\author[MVW]{Magdalina von Wunsch-Rolshoven}

\address{Fakult\"at f\"ur Mathematik, Universit\"at Regensburg, Germany}
\email{stefan.bastl@stud.uni-regensburg.de}
\email{alison.durst@stud.uni-regensburg.de}
\email{sfriedl@gmail.com}
\email{tobias.hirsch@stud.uni-regensburg.de}
\email{clara.loeh@ur.de}
\email{lars.munser@ur.de}
\email{patrick.perras@stud.uni-regensburg.de}
\email{lisa.schambeck@stud.uni-regensburg.de}
\email{matthias.uschold@ur.de}

\address{School of Mathematics and Physics, The University of Queensland, Brisbane QLD 4072, Australia}
\email{rhuaidi.burke@uq.edu.au}

\address{Mathematisches Institut, Universit\"at zu K\"oln, Weyertal 86-90, 50931 k\"oln, Germany}
\email{rchatt@math.uni-koeln.de, rchattmath@gmail.com}

\address{Department of Mathematics, Indian Institute of Technology Palakkad, Kerala, India}
\email{subhankardey@iitpkd.ac.in}

\address{Max Planck Institute for Mathematics, Bonn, Germany}
\email{galvin@mpim-bonn.mpg.de}
\email{truoel@mpim-bonn.mpg.de, paulagtruoel@gmail.com}

\address{Universit\"at Bonn, Regina-Pacis-Weg 3, Bonn, Germany}
\email{alexgarciarivas@hotmail.es}
\email{s6cahobo@uni-bonn.de}
\email{s6frkern@uni-bonn.de}
\email{s6shleee@math.uni-bonn.de }
\email{vesela@math.uni-bonn.de }
\email{s6meweis@math.uni-bonn.de}
\email{s6mmvonw@uni-bonn.de}

\address{Humboldt-Universit\"at zu Berlin, Rudower Chaussee 25, 12489 Berlin, Germany.}
\email{chun-sheng.hsueh@hu-berlin.de}
\email{naageswaran.manikandan@hu-berlin.de}
\email{mousseal@hu-berlin.de}
\email{suchodod@hu-berlin.de}
\email{annika.thiele@hu-berlin.de}

\address{Universidad de Sevilla, Dpto.\ de Álgebra,
Avda.\ Reina Mercedes s/n,
41012 Sevilla, Spain}
\email{mkegel@us.es, kegelmarc87@gmail.com}

\address{School of Mathematics and Statistics, University of Glasgow, United Kingdom}
\email{m.pencovitch.1@research.gla.ac.uk}
\email{mark.powell@glasgow.ac.uk}

\address{Institut f\"ur Mathematik IMa, Im Neuenheimer Feld 205, 69120 Heidelberg, Germany}
\email{jquintanilha@mathi.uni-heidelberg.de}

\address{Department of Applied Mathematics, Charles University, Prague, Czech Republic}
\email{tancer@kam.mff.cuni.cz}

\def\subjclassname{\textup{2020} Mathematics Subject Classification}
\expandafter\let\csname subjclassname@1991\endcsname=\subjclassname



\begin{abstract}
We show that there exists an algorithm that takes as input two closed, simply connected, topological $4$-manifolds and decides whether or not these $4$-manifolds are homeomorphic. In particular, we explain in detail how closed, simply connected, topological $4$-manifolds can be naturally represented by a Kirby diagram consisting only of $2$-handles. This representation is used as input for our algorithm. Along the way, we develop an algorithm to compute the Kirby--Siebenmann invariant of a closed, simply connected, topological $4$-manifold from any of its Kirby diagrams and describe an algorithm that decides whether or not two intersection forms are isometric. 

In a slightly different direction, we discuss the decidability of the stable classification of smooth manifolds with more general fundamental groups. Here we show that there exists an algorithm that takes as input two closed, oriented, smooth $4$-manifolds with fundamental groups isomorphic to a finite group with cyclic Sylow $2$-subgroup, an infinite cyclic group, or a group of geometric dimension at most $3$ (in the latter case we additionally assume that the universal covers of both $4$-manifolds are not spin), and decides whether or not these two $4$-manifolds are orientation-preserving stably diffeomorphic.
\end{abstract}

\maketitle
\markboth{\shorttitle}{\shorttitle}

\section{Introduction}
\addtocontents{toc}{\protect\setcounter{tocdepth}{1}}

It is a consequence of the resolution of the geometrization conjecture~\cite{perelman1,perelman2} that the homeomorphism problem is solved for closed, orientable manifolds of dimension less than or equal to~$3$, see for example~\cite{Kuperberg}. On the contrary, a well-known theorem due to Markov states that there exists no algorithm that takes as input two $4$-manifolds (presented as triangulations) and outputs in finite time whether or not these manifolds are homeomorphic~\cite{Markov} (see also~\cite{Stanko, Chernavsky,Kirby_Markovsthm,Gordon,Tancer}). The idea to prove this theorem (and related results for manifolds of dimension larger than $4$) is to produce from such a potential algorithm a solution to an undecidable problem in group theory. This crucially uses the fact that the fundamental group of the input manifolds can be isomorphic to any finitely presented group. Thus it seems natural to ask whether such algorithms exist if the isomorphism type of the fundamental groups of the input manifolds is fixed. The main result of this article is such an algorithm for simply connected $4$-manifolds.

\begin{reptheorem}{thm:main}[Abbreviated version]
	There exists an algorithm that 
    \begin{itemize}
        \item takes as input two closed, oriented, simply connected, topological $4$-manifolds $X$ and $X'$, presented as Kirby diagrams $($we refer to Section~\ref{sec:Kirby_into} for the explanation of what a Kirby diagram in this setting means$)$, and 
        \item outputs whether or not $X$ and $X'$ are orientation-preserving homeomorphic. 
    \end{itemize}
\end{reptheorem}

Our proof of \Cref{thm:main} is based on a careful step-by-step analysis of Freedman's classification~\cite{Freedman} of closed, oriented, simply connected, topological $4$-manifolds by the intersection form and the Kirby--Siebenmann invariant. In particular, we provide algorithms for computing and comparing these invariants from the input data.

\begin{remark}
    If $X$ and $X'$ are orientable but unoriented $4$-manifolds, and we wish to determine whether or not they are homeomorphic, it suffices to choose orientations on each, and then run our algorithm with input $(X,X')$ and $(X,-X')$. If at least one of the answers is yes, the manifolds are homeomorphic. If both answers are no, the manifolds are not homeomorphic. 
\end{remark}


\subsection{Kirby diagrams of simply connected topological 4-manifolds}\label{sec:Kirby_into}

Since there exist closed, oriented, simply connected, topological $4$-manifolds that admit neither triangulations nor smooth structures and thus do not admit a handle decomposition, it is a priori not clear how to input topological $4$-manifolds into an algorithm. 

To encode a closed, oriented, simply connected, topological $4$-manifold~$X$ as a finite data set, 
in \cref{sec:input} we introduce  the notion of a \emph{Kirby diagram} of $X$. For this, let $L$ be a framed link in $S^3$ with unimodular linking matrix. By attaching $4$-dimensional $2$-handles to $D^4$ along $L$, we obtain the diffeomorphism type of a compact, oriented, smooth $4$-manifold $W_L$ whose boundary $Y$ is an integral homology $3$-sphere. From Freedman's work~\cite{Freedman} it follows that there exists a contractible, compact, topological $4$-manifold $C$ with boundary $Y$, which we can glue to $W_L$ to obtain a topological $4$-manifold~$X_L$; see \Cref{fig:splitting_intro} for a schematic of the situation. The following theorem shows
that $X_L$ yields the homeomorphism type of a closed, oriented, simply connected, topological $4$-manifold and that, conversely, any such manifold arises via this construction.

\begin{reptheorem}{thm:topKirby}
    Let $L$ be a framed link with
    unimodular linking matrix.
    \begin{enumerate}
        \item Then $X_L$ is a closed, oriented, simply connected, topological $4$-manifold.
        \item The oriented homeomorphism type of $X_L$ only depends on the isotopy class of $L$.
        \item\label{item:topKirby-3} For every closed, oriented, simply connected, topological $4$-manifold $X$ there exists a framed link $L$ with unimodular linking matrix such that $X_L$ is orientation-preserving homeomorphic to $X$.
    \end{enumerate}
\end{reptheorem}

In the situation of \cref{thm:topKirby}~\eqref{item:topKirby-3}, we call $L$ a \textit{Kirby diagram} of the topological $4$-manifold $X$. 

\begin{figure}[htbp]
    \includegraphics{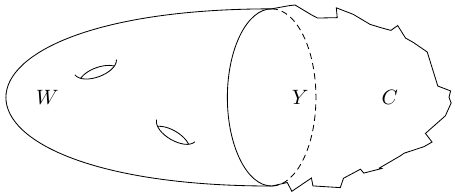}
    \caption{Splitting $X$ along $Y$, where $W$ is smooth and $C$ is contractible.}
    \label{fig:splitting_intro}
\end{figure}

\begin{remark}
    Our approach to present closed, oriented, simply connected, topological $4$-manifolds crucially uses the simple connectivity 
    of the manifold for obtaining the decomposition into a smooth piece and a contractible piece. On the other hand, there is also work by Freedman--Zuddas~\cite{Freedman_Zuddas} on presenting general compact (in particular non-simply connected) topological $4$-manifolds via finite data.
\end{remark}


\subsection{The Kirby--Siebenmann invariant}
The \textit{Kirby-Siebenmann invariant} $\ks(X)$ of a topological manifold $X$ (see \cite[Section 10.2B]{FQ} and
\cite[p.~318]{Kirby_Siebenmann}) 
is an element of $H^4(X;\Z/2)$. 
It vanishes if $X$ admits a smooth or a PL structure. Thus, if $X$ is a connected, closed,  topological $4$-manifold, the Kirby--Siebenmann invariant is either $0$ or $1$. In \cref{section:computing-Kirby-Siebenmann} we will use a theorem of Freedman--Kirby (see \cref{theorem:Freedman-Kirby}) to show that the Kirby--Siebenmann invariant is computable.

\begin{repproposition}{lem:Ks_algo}
   There exists an algorithm that 
   \begin{itemize}
       \item takes as input a framed link $L$ with unimodular linking matrix, and
       \item outputs the Kirby--Siebenmann invariant of $X_L$.
   \end{itemize}
\end{repproposition}


\subsection{Intersection forms}

One of the most important invariants in $4$-manifold theory is the intersection form. In \cref{section:background_interesection-forms} we will recall the background on the intersection form and in \cref{sec:int_form} we will see that the intersection form of a closed, oriented, simply-connected, topological $4$-manifold~$X_L$, presented by a framed link $L$, is in a preferred basis given by the linking matrix of $L$. In \cref{section:intersection-form-algorithm} we will present a purely algebraic algorithm to compare two such intersection forms.

\begin{repproposition}{prop:intersectionform}
  There exists an algorithm that 
   \begin{itemize}
       \item takes as input two integral, symmetric, unimodular matrices $V$ and $V'$, and 
       \item outputs whether or not these two matrices are congruent over $\Z$.
   \end{itemize} 
\end{repproposition}


\subsection{Algorithms for the stable classification of 4-manifolds}

We say that two smooth $4$-manifolds $X_1$ and $X_2$ are \textit{stably diffeomorphic} if there exist natural numbers $k_1$ and $k_2$ such that $X_1\#_{k_1}(S^2\times S^2)$ is diffeomorphic to $X_2\#_{k_2}(S^2\times S^2)$.  
Note that by \cite{Gompf1985} two closed, orientable, smooth $4$-manifolds are stably diffeomorphic if and only if they are stably homeomorphic. 
The classification of $4$-manifolds up to stable diffeomorphism is coarser, and hence easier to compute, 
which means that we can leave the realm of simply connected $4$-manifolds. In \cref{sec:stable_algo} we will use the classification of $4$-manifolds up to stable diffeomorphism to present the following two algorithms. 

\begin{reptheorem}{thm:mainAlmostSPinOrSpin}
    There exists an algorithm that 
    \begin{itemize}
        \item takes as input oriented triangulations of closed, oriented, smooth $4$-manifolds $X_1$ and $X_2$ such that either 
        \begin{enumerate}
            \item their fundamental groups are both isomorphic to the infinite cyclic group,
            \item or $X_1$ has a finite fundamental group with a cyclic Sylow $2$-subgroup and $X_2$ has a finite fundamental group, and
        \end{enumerate}
        \item outputs whether or not $X_1$ and $X_2$ are orientation-preserving stably diffeomorphic. 
    \end{itemize}
\end{reptheorem}

\begin{reptheorem}{thm:mainNon-Spin}
    There exists an algorithm that 
    \begin{itemize}
        \item takes as input oriented triangulations of closed, oriented, smooth $4$-manifolds $X_1$ and $X_2$, such that both universal covers $\widetilde{X}_1$ and $\widetilde{X}_2$ are not spin and such that
        \begin{enumerate}
            \item their fundamental groups are isomorphic and of homological dimension $\leq 3$,
            \item or their fundamental groups are both finite, and
        \end{enumerate}
        \item outputs whether or not $X_1$ and $X_2$ are orientation-preserving stably diffeomorphic. 
    \end{itemize}  
\end{reptheorem} 

\begin{remark}
    In the literature one also finds stable diffeomorphism defined with $k_1=k_2$. We could have also used this definition, but then would have had to require that $\chi(X_1) = \chi(X_2)$. Since the Euler characteristic is straightforward to algorithmically compute from a triangulation, \cref{thm:mainAlmostSPinOrSpin,thm:mainNon-Spin} also hold with the more rigid alternative definition of stable diffeomorphism. 
\end{remark}

\subsection{Open questions}

We end this introduction by highlighting some natural related questions.

In recent years, there have been several attempts to find the smallest (in terms of second homology) unrecognisable simply connected, closed, topological $4$-manifold.
The strongest result so far was obtained by Tancer~\cite{Tancer} who showed  that $M_9 = \sumStwo{9}$ is unrecognisable among closed, triangulated, topological $4$-manifolds. 
Earlier results of Shtan'ko~\cite{Stanko} 
and Gordon~\cite{Gordon} showed that $M_k$ is unrecognisable for $k=14$ and $k=12$, respectively; see also \cite{Chernavsky}. The recognisability of the simplest closed, topological $4$-manifold, the $4$-sphere $S^4$, is still unknown, \ie the following remains open, see for example~\cite{Weinberger,Kirby_Markovsthm,Gordon,Tancer}. 

\begin{question}\label{conj:S4_intro}
Does there exist an algorithm that 
    \begin{itemize}
        \item takes as input a triangulation of a closed $4$-manifold $X$, and 
        \item outputs whether or not $X$ is homeomorphic to $S^4$?
    \end{itemize}
\end{question}
    
On the other hand, one could ask for which other fundamental groups there is an algorithmic classification of $4$-manifolds with this fundamental group,
similar to \cref{thm:main}, \cref{thm:mainAlmostSPinOrSpin}, or \cref{thm:mainNon-Spin}. Similar to our approach in \Cref{thm:main}, one could fix a nontrivial group~$\pi$ and study the homeomorphism problem for 
closed, oriented, topological $4$-manifolds with fundamental group isomorphic to $\pi$; or with fundamental group in a well-understood subclass of finitely presented groups. For some groups $\pi$, complete sets of invariants have been found which classify closed, orientable, topological $4$-manifolds with fundamental group isomorphic to $\pi$, most notably $\Z$~\cite{FQ}, $\Z/n$~\cite{Hambleton-Kreck:1988-1,Hambleton-Kreck-93}, and solvable Baumslag--Solitar groups~\cite{HKT}. 
Conversely, it would be interesting to see if there exists a group such that there exists no algorithmic classification of $4$-manifolds with that fundamental group.

\begin{question}\label{ques:word_problem}
    Does there exist a finitely presented group $\pi$, such that there is \textbf{no} algorithm that
    \begin{itemize}
        \item takes as input two $4$-manifolds $X_1$ and $X_2$, presented as triangulations, with fundamental groups isomorphic to $\pi$, and
        \item outputs whether or not $X_1$ and $X_2$ are homeomorphic?
    \end{itemize}
\end{question}

It would also be interesting to study $4$-manifolds from an algorithmic viewpoint with respect to smooth aspects. Here we mention the following two problems.

\begin{question} \label{ques:smooth}
    Does there exist an algorithm that
    \begin{itemize}
        \item takes as input a topological $4$-manifold $X$ (if $X$ is closed, oriented, and simply connected, we can use a Kirby diagram of $X$ as input, otherwise we could use~\cite{Freedman_Zuddas}), and
        \item outputs whether or not $X$ admits a smooth structure?
    \end{itemize}
\end{question}

\begin{remark}
   A positive resolution of the $11/8$-conjecture (see~\cite[Page 16]{GS}) would imply a positive answer to \cref{ques:smooth} for simply connected manifolds.
\end{remark}

\begin{question} \label{ques:smooth2}
    Does there exist an algorithm that
    \begin{itemize}
        \item takes as input two smooth, homeomorphic $4$-manifolds $X_1$ and $X_2$, presented as triangulations, and
        \item outputs whether or not $X_1$ and $X_2$ are diffeomorphic?
    \end{itemize}
\end{question}

We remark that \cref{conj:S4_intro} is known to be false in dimensions larger than $4$~\cite{VKF}, whereas the analogues of \cref{ques:smooth,ques:smooth2} seem to be open in all dimensions larger than $4$. Question~\ref{ques:word_problem} and a theorem similar to \cref{thm:main} for \textit{PL} and \textit{smooth} manifolds of dimension $5$ or larger are mentioned in~\cite{Nabutovsky_Weinberger}, building upon the computability of higher homotopy groups~\cite{Brown}.

\subsection*{A note about algorithms}

We take a pragmatic approach to the notion of algorithms. (The underlying precise notion is in terms of Turing machines~\cite{boolos_burgess_jeffrey} or equivalent models of computation.) For each algorithm, we will explain how the input is represented as finite data. We will then describe the algorithms in standard mathematical language, giving enough details that the translation to a concrete algorithmic setting is evident.


\subsubsection*{Acknowledgements}
 This paper reports on the talks and discussion sessions during the workshop on \textit{Algorithms in $4$-manifold topology}~\cite{workshop}, hosted at Universit\"at Regensburg in September 2024, and organised by Stefan Friedl, Marc Kegel, and Birgit Tiefenbach. 
 The authors of this paper are the workshop participants. We are all grateful to the organisers for bringing us together, and to each other for helpful discussions and suggestions, and for an unmatched collaborative spirit. We thank Jonathan Bowden, Duncan McCoy, and Filip Misev for useful comments and discussion. 

\subsubsection*{Individual grant support}
We gratefully acknowledge financial support of the SFB~1085 \emph{Higher Invariants} (Universit\"at Regensburg, funded by the DFG, ID 224262486). 
Rima Chatterjee was supported by SFB/TRR 191 \textit{Symplectic Structures in Geometry, Algebra and Dynamics}, funded by the DFG (ID 281071066 - TRR 191).
Subhankar Dey was supported by an individual research grant of the DFG (ID 505125645).
Daniel Galvin and Paula Truöl would like to thank the \textit{Max Planck Institute for Mathematics in Bonn} for their hospitality and financial support.
Alejandro García Rivas was supported by MSCA-DN CaLiForNIA - 101119552.
Chun-Sheng Hsueh and Naageswaran Manikandan were supported by \textit{The Berlin Mathematics Research Center MATH+} funded by the DFG (EXC-2046/1, ID 390685689).
Marc Kegel is supported by the DFG, German Research Foundation, (Project: 561898308); by a Ram\'on y Cajal grant (RYC2023-043251-I) and the project PID2024-157173NB-I00 funded by MCIN/AEI/10.13039/501100011033, ESF+ and FEDER, EU; and by a VII Plan Propio de Investigación y Transferencia (SOL2025-36103) of the University of Sevilla.
Mark Powell was funded by EPSRC New Investigator grant \textit{Classifying $4$-manifolds} (EP/T028335/2). 
Martin Tancer was supported by the ERC-CZ project LL2328 of the Ministry of Education of Czech Republic.
Simona Veselá was supported by the DFG (EXC-2047/1, ID 390685813).

\subsubsection*{Open access}
For the purpose of open access, the authors have applied a CCBY license to this version, and will do the same to any author-accepted manuscript arising from this submission. 


\section{Intersection forms}\label{section:background_interesection-forms}

Recall that for a compact, oriented, connected $4$-manifold $X$, the \textit{intersection form} of $X$ is defined as 
\begin{align*} \lambda_X\colon H_2(X)\times H_2(X)&\longrightarrow  \Z\\
(\varphi,\psi)&\longmapsto  \langle\operatorname{PD}_X^{-1}(\varphi)\cup 
\operatorname{PD}_X^{-1}(\psi),[X]\rangle,\end{align*}
where $\operatorname{PD}_X \colon H^2(X,\partial X;\Z) \to H_2(X)$ is the Poincar\'e duality isomorphism given by capping with the fundamental class~$[X] \in H_4(X,\partial X)$ determined by the orientation.  
Note that the intersection form is symmetric and bilinear.
It follows from the Hurewicz theorem, Poincar\'e duality, and the universal coefficient theorem that if $X$ is simply connected, then $H_2(X)$ is torsion-free.
Finally note that it follows from Poincar\'e duality that 
the intersection form is unimodular if and only if 
 $X$ is closed  or if  $\partial X$  is an integral homology $3$-sphere. 

If we choose a basis of $H_2(X)$, which induces an isomorphism $f\colon H_2(X)\rightarrow \Z^m$, then under this isomorphism the intersection form can be represented by an integral, symmetric, unimodular\footnote{Here we say that a matrix is \textit{integral} if it has only integer coefficients and such a matrix is called \textit{unimodular} if it is a square matrix and has determinant $\det(V)=\pm1$.} matrix~$V$, via 
\begin{align*}
    \lambda_X\colon \Z^m \times \Z^m&\longrightarrow \Z\\
    (a,b)&\longmapsto a^TVb.
\end{align*}
Conversely, any symmetric $(m\times m)$-matrix $V$ with integral entries and determinant $\det(V)=\pm1$ represents a unimodular, symmetric bilinear form $\Z^m\times \Z^m\rightarrow \Z$ via the above formula. Two such matrices $V$ and $V'$ represent isometric bilinear forms if and only if they are congruent over~$\Z$, i.e.\ there exists an integral, unimodular matrix $P$ such that $V=P^TV'P$. 

The following theorem of  Freedman~\cite{Freedman} implies in particular that every symmetric, unimodular, bilinear form over the integers is isometric to the intersection form of some closed, oriented, simply connected, topological $4$-manifold.

\begin{theorem}[{Freedman~\cite{Freedman}}]\label{thm:freedman}
\mbox{}
\begin{enumerate}
\item Let $\lambda$ be a symmetric, unimodular, bilinear form over $\Z$ and let $k\in \Z/2$. 
$($If $\lambda$ is even, then we assume that $k\equiv \frac{1}{8}\cdot \operatorname{sign}(\lambda)\, \operatorname{mod} 2$.$)$
There exists a closed, oriented, simply connected, topological $4$-manifold $X$ whose intersection form is isometric to $\lambda$ and with Kirby--Siebenmann invariant $\operatorname{ks}(X)=k\in \Z/2$.
\item Let $X$ and $Y$ be two closed, oriented, simply connected, topological $4$-manifolds
with equal Kirby--Siebenmann invariants $\operatorname{ks}(X)=\operatorname{ks}(Y)\in \Z/2$. If $f\colon ({H}_2(X),\lambda_X)\to ({H}_2(Y),\lambda_Y)$ is an isometry of the intersection forms, then there exists an orientation-preserving homeomorphism
$h\colon X\to Y$ such that $h_*=f$. Furthermore, $f$ is unique up to isotopy.\qed
\end{enumerate}
\end{theorem}
 For more background on intersection forms on $4$-manifolds we refer for example to~\cite[Chapter 1.2]{GS}, \cite[Chap. II]{Kirby}, \cite[Chap. III]{Scorpan}, or \cite[Chap. II, \S 2.7]{Prasolov}.


\section{Kirby diagrams of topological 4-manifolds}\label{sec:input}

Natural ways to present compact, smooth $4$-manifolds are via triangulations or handle decompositions (Kirby diagrams). However, it is known that a compact $4$-manifold admits a triangulation if and only if it admits a handle decomposition, which in turn is equivalent to 
admitting a smooth structure~\cite{Whitehead,Hirsch_Mazur, munkres_1, munkres_2, cerf}, cf.\ \cite[Theorem 8.3B]{FQ}.
So if we want to describe a topological $4$-manifold (that possibly carries no smooth structure), we cannot use any of the above presentation methods. In this section, we explain how to specify a closed, oriented, simply connected, topological $4$-manifold $X$ by a Kirby diagram (even if it does not carry a smooth structure and thus does not admit a handle decomposition). The idea is to decompose $X$ as $X=W\cup_Y C$, where $W$ is a codimension-zero submanifold of $X$ that does carry a smooth structure and which can be described by a Kirby diagram and  $C$ is a contractible topological $4$-manifold such that the intersection $Y=W \cap C$ is an integral homology $3$-sphere.  In the following, we make this precise.

\begin{definition}\label{def:framedlink}
A \emph{framed link} is a closed,  oriented, smooth, $1$-dimensional submanifold $L = L_1 \sqcup \cdots \sqcup L_m$ in~$S^3$ with ordered components, together with a $\Z$-label of each component~$L_i$, called its \textit{framing}. 
\end{definition}

This framing corresponds to a homotopy class of trivialisations of the normal bundle of~$L_i$, for each~$i$. Each component of~$L$ has a closed tubular neighbourhood~${\nu} L_i$, and using the framing we obtain the isotopy class of a diffeomorphism from $S^1 \times D^2$ to ${\nu} L_i$ (using the standard convention that the $0$-framing corresponds to the Seifert framing of $L_i$), we refer for example to~\cite{GS} for details. 

\begin{remark}\label{rem:linkrep}
  In algorithms, framed links will be represented by $\Z$-labelled, oriented, ordered link diagrams. There are several known ways to represent link diagrams (up to planar isotopy), for example via PD codes~\cite{knotatlas,Mast}, DT codes~\cite{dt_codes}, Gau\ss\ codes~\cite{KnotInfo}, braid words~\cite{Artin}, or isosignatures~\cite{regina}.
\end{remark}

\begin{definition}\label{def:unimodular} \hfill
\begin{enumerate}
\item 
Let $L$ be a framed link with components $L_1\sqcup\cdots\sqcup L_m$ and framings $f_1,\ldots,f_m$. Then the \textit{linking matrix} $V=(v_{ij})_{1\leq i,j\leq m}$ of $L$ is defined by $v_{ii}:= f_i$ and $v_{ij}:=  \lk(L_i,L_j)$ if $i\neq j$. 
\item 
  A \emph{framed, unimodular link} is an oriented, framed link whose 
  linking matrix~$V$ is unimodular.
  \end{enumerate}
\end{definition}

\begin{remark} \label{rem:W_L} We make a few remarks about this definition.
\begin{enumerate}
    \item The linking matrix of a framed link~$L$ can be computed algorithmically from each diagram of~$L$ (see the proof of \cref{lem:computing_intersection_from}). Since the determinant of a matrix with integral coefficients is computable, it is algorithmically decidable whether a given framed link is unimodular. 
    \item Any framed link $L$ determines the diffeomorphism type of a compact, oriented, simply connected, smooth $4$-manifold $W_L$, obtained by attaching $2$-handles to $D^4$ along~$L$. Then $H_2(W)$ admits a preferred basis, whose $i$-th element is represented by the core of the $2$-handle attached to the $i$-th component~$L_i$, union a cone on $L_i$ in $D^4$. In this basis, the intersection form~$\lambda_{W_L}$ is represented by the linking matrix of $L$. It follows that if the framed link~$L$ is unimodular, then the intersection form $\lambda_{W_L}$ is unimodular, and thus $\partial W_L$ is an integer homology $3$-sphere. We refer to~\cite[Chapter 4]{GS} for more details.
\end{enumerate}
\end{remark}

Next, we observe that from a framed, unimodular link we can build a topological $4$-manifold.

\begin{definition}\label{def:top_Kirby}
Let $L$ be a framed, unimodular link. We use it to build a topological $4$-manifold~$X_L$ as follows.
    \begin{enumerate}
    \item We choose a compact, oriented, simply connected, smooth $4$-manifold~$W_L$ obtained by attaching $2$-handles to $D^4$ along~$L$ as explained in \cref{rem:W_L}.
    \item The boundary $\partial W_L$ of $W_L$ is an integer homology $3$-sphere and thus, by Freedman~\cite[Theorem~1.4']{Freedman}, there exists a contractible, compact, oriented, topological $4$-manifold $C$ whose boundary is orientation-reversing homeomorphic to $\partial W_L$. 
    \item We choose an orientation-reversing homeomorphism $h\colon \partial W_L\to \partial C$ and use this homeomorphism to glue $W_L$ and $C$, i.e.\ we build 
    \begin{equation*}
        X_L:=W_L\cup_h C.
    \end{equation*}
\end{enumerate}
\end{definition}

The following theorem, which is the main result of this section, states among other things that the homeomorphism type of $X_L$ is independent of all choices made throughout this construction. 

\begin{theorem}\label{thm:topKirby}
    Let $L$ be a framed, unimodular link.
    \begin{enumerate}
        \item $X_L$ is a closed, oriented, simply connected, topological $4$-manifold.
        \item The oriented homeomorphism type of $X_L$ only depends on the isotopy class of $L$ as framed link.
        \item For every closed, oriented, simply connected, topological $4$-manifold $X$ there exists a framed, unimodular link $L$ such that $X_L$ is orientation-preserving homeomorphic to $X$.
    \end{enumerate}
\end{theorem}

\begin{definition}
 In the situation of \cref{thm:topKirby}(3), we say that $L$ is a \textit{Kirby diagram} of the topological $4$-manifold $X$. We call~$W_L$ the \textit{handle part}, $C$ the \textit{contractible part}, and $Y:=\partial C=\partial{W_L}$ the \emph{splitting $3$-manifold} of~$X_L$.   
\end{definition}

 We refer to \cref{fig:splitting_intro} for a schematic sketch of the situation. 

\begin{remark}
     We make a comment about the justification of the term \textit{Kirby diagram} for the above construction. Recall from \cref{rem:W_L} that a framed link $L$ determines the diffeomorphism type of a compact, oriented, smooth $4$-manifold $W_L$ by attaching $2$-handles along $L$ to $D^4$. 
     
     If the boundary $\partial W_L$ is diffeomorphic to $\#_m (S^1\times S^2)$ then we can glue in a copy of $\natural_m (S^1\times D^3)$ to obtain a closed, oriented, smooth $4$-manifold $W$. A result by Laudenbach--Poenaru~\cite{LP} implies that the diffeomorphism type of this closed $4$-manifold $W$ is independent of the gluing map and thus $L$ is called Kirby diagram of $W$. 

    If the boundary $\partial W_L$ is diffeomorphic to an integral homology $3$-sphere we can glue in a contractible topological $4$-manifold to obtain a closed manifold. We can use Freedman~\cite{Freedman} to show that the homeomorphism type of that closed $4$-manifold is independent of that gluing map. Thus the construction is essentially the same, and the name \textit{Kirby diagram} for this topological $4$-manifold seems to be justified.
\end{remark}

\cref{thm:topKirby} is proven via the following sequence of lemmas.

\begin{lemma} \label{lem:simply_connected}
    Let $L$ be a framed, unimodular link. Then $X_L$ is a closed, oriented, simply connected, topological $4$-manifold.
\end{lemma}

\begin{proof}
    Since $X_L = W_L \cup_h C$ consists of two compact, oriented, simply connected $4$-manifolds glued along their common boundary by an orientation-reversing homeomorphism, $X_L$ is a closed, oriented $4$-manifold. The Seifert--van Kampen theorem implies that $X_L$ is also simply connected.
\end{proof}

\begin{lemma}\label{lemma-inclusion-W-iso-on-H2}
    Let $L$ be a framed, unimodular link. Then the inclusion map $\iota \colon W_L\hookrightarrow X_L = W_L \cup_h C$ induces an isomorphism $\iota_* \colon H_2(W_L)\to H_2(X_L)$. Thus the handle part $W_L$ determines a preferred isomorphism $f \colon H_2(X_L) \to \Z^m$. 
\end{lemma}

\begin{proof}
    The first claim follows from the Mayer--Vietoris sequence of $X_L = W_L \cup_Y C$, using that $H_2(Y)$, $H_1(Y)$, and $H_2(C)$ are all trivial groups. As explained in \cref{rem:W_L} there exists a preferred isomorphism $H_2(W_L)\to\Z^m$ which thus induces a preferred isomorphism $f \colon H_2(X_L) \to \Z^m$.
\end{proof}

\begin{lemma}\label{lem:uniqueness}
    Let $L$ be a framed, unimodular link. If $X_L = W_L \cup_h C$ and $X'_L= W'_L \cup_{h'} C'$ are two topological $4$-manifolds constructed as in \cref{def:top_Kirby}, then there exists an orientation-preserving homeomorphism $g \colon X_L \to X'_L$. Moreover, if we denote by $f \colon H_2(X_L) \to \Z^m$  and $f' \colon H_2(X'_L) \to \Z^m$ the preferred isomorphisms constructed in \cref{lemma-inclusion-W-iso-on-H2}, we can choose $g$ such that 
    $f' \circ g_*\circ f^{-1} = \Id_{\Z^m}$. 
\end{lemma}

\begin{proof}
First, recall that by \cref{rem:W_L}~(2),  there exists an orientation-preserving diffeomorphism $g^W\colon W_L\to W'_L$, and its induced isomorphism $g^W_* \colon H_2(W_L) \to H_2(W'_L)$ sends the bases determined by~$L$ to one another.
    Moreover, these bases determine our preferred bases for $H_2(X_L)$ and $H_2(X'_L)$ via the inclusion-induced isomorphisms $\iota_* \colon H_2 (W_L) \to H_2(X_L)$ and $\iota_*' \colon H_2(W'_L) \to H_2(X'_L)$ from \cref{lemma-inclusion-W-iso-on-H2}. Therefore, we have 
    \begin{equation*}
        f' \circ \iota'_* \circ g^W_* \circ \iota_*^{-1} \circ f^{-1} = \Id_{\Z^m}.
    \end{equation*} 
    Thus, to prove the lemma, it is enough to show that the diffeomorphism $g^W\colon W_L\to W'_L$ extends to an orientation-preserving homeomorphism $g\colon X_L\to X'_L$. 
    To see this, consider the $4$-manifold $C \cup -C'$ glued via the homeomorphism $h'\circ g^W\circ h^{-1}$. 
    Using the fact that  $C$ and $C'$ are contractible and using Seifert-van Kampen and Mayer-Vietoris one can easily verify that 
    $\pi_1(C \cup -C')=0$ and $H_2(C \cup -C')=0$.
    It follows from Theorem~\ref{thm:freedman}  that $C \cup -C'$ is homeomorphic to $S^4$. Then $D^5$ gives an $h$-cobordism rel.\ boundary $(D^5;C,C')$ from $C$ to $C'$. This is an $h$-cobordism because $C$, $C'$ and~$D^5$ are all contractible, so the inclusion maps are necessarily homotopy equivalences. By the $h$-cobordism theorem~\cite[Theorem~7.1A]{FQ}, this cobordism is a product, i.e.\ we have a homeomorphism of pairs $(D^5,C') \cong (C' \times [0,1],C' \times \{1\})$ that is the identity on $C'$. Restricting this to $C \subseteq D^5$ yields an orientation-preserving homeomorphism $C \to C' \times \{0\} = C'$, which extends $g^W$ to an orientation-preserving homeomorphism $g\colon X_L \to X'_L$. 
\end{proof}

\begin{lemma}\label{lem:existence}
  Let $X$ be a closed, oriented, simply connected, topological $4$-manifold. Then there exists 
  a framed, unimodular link~$L$ such that $X_L$ is orientation-preserving homeomorphic to~$X$.
\end{lemma}

\begin{proof}
Let $V$ be a matrix representing the intersection form $\lambda_X$ of $X$. Since $X$ is a closed manifold, \(V\) is an integral, symmetric, unimodular matrix. Thus we can choose a framed, unimodular link~$L$ in $S^3$ whose linking matrix is $V$. Then the intersection form of $X_L$ is represented by \(V\) (see Lemma~\ref{lem:linking_equal_intersection} below and~\cite[Proposition~4.5.11]{GS}) and in particular the intersection forms of $X_L$ and $X$ are isometric. In~\cite[p.~371]{Freedman}, Freedman explained how to alter $L$ by tying a trefoil into some specified component of $L$ (belonging to the characteristic sublink), in order to change the Kirby--Siebenmann invariant $\ks(X_L)$ of $X_L$  whilst preserving its intersection form. Thus we can also arrange that $\ks(X_L) = \ks(X)$. 
For more details on this construction, we refer to the algorithm for computing the Kirby--Siebenmann invariant in \cref{lem:Ks_algo} below.  
By the uniqueness part Freedman's Theorem~\ref{thm:freedman} we conclude that \(X_L\) is orientation-preserving homeomorphic to \(X\).
\end{proof}

Putting these lemmas together yields a proof of \cref{thm:topKirby}.

\begin{proof}[Proof of \cref{thm:topKirby}]
(1) is \cref{lem:simply_connected}, (2) follows from \cref{lem:uniqueness}, and (3) is \cref{lem:existence}.
\end{proof}


\section{Extracting the intersection form}\label{sec:int_form}
In this section, we discuss how to compute the intersection form of a closed, oriented, simply connected, topological $4$-manifold from any of its Kirby diagrams.

\begin{lemma}\label{lem:linking_equal_intersection}
     Let $L$ be a framed, unimodular link. Then the linking matrix~$V$ of $L$ represents the intersection form $\lambda_{X_L}$ of $X_L$ in the preferred basis induced by the handle part.
\end{lemma}

 \begin{proof}
We write $X_L$ as $X_L=W_L\cup_h C$, with $C$ a contractible $4$-manifold and $W$ the handle part. 
     By \cref{lemma-inclusion-W-iso-on-H2}, the inclusion $\iota \colon W_L\hookrightarrow X_L$ induces an isomorphism $\iota_* \colon H_2(W)\to H_2(X)$, so it follows from naturality of the cup and cap products that $\iota_*$ is an isometry from the intersection form~$\lambda_{W_L}$ of $W_L$, to~$\lambda_{X_L}$  (we refer to \cite[Chapter~153.7]{friedlmonster} for details). As observed in \cref{rem:W_L} (see also~\cite[Proposition~4.5.11]{GS}), the form~$\lambda_{W_L}$ is given, with respect to the basis induced from the $2$-handles of~$W_L$, by~$V$.
 \end{proof}

\begin{lemma}\label{lem:computing_intersection_from}
    There exists an algorithm that
    \begin{itemize}
        \item takes as input a framed, unimodular link $L$, and
        \item outputs the matrix $V$ that represents the intersection form of $X_L$ in the preferred basis induced by the handle part.
    \end{itemize}
\end{lemma}

\begin{proof}
   Since the linking numbers of the components of an oriented link $L$ are computable from any link diagram of $L$ via its combinatorial formula (see for example~\cite[Proposition~4.5.2]{GS}), we can algorithmically compute the linking matrix $V$ of a framed link $L$. 
   By Lemma~\ref{lem:linking_equal_intersection}, the linking matrix represents the isometry type of the intersection form of $X_L$ in the desired basis.
\end{proof}


\section{Computing the Kirby--Siebenmann invariant}\label{section:computing-Kirby-Siebenmann}

In this section, we explain how to compute the Kirby--Siebenmann invariant of a closed, oriented, simply connected, topological $4$-manifold from any of its Kirby diagrams.
For that, we first discuss the Arf invariant of a characteristic homology class, which will appear in a formula for the Kirby--Siebenmann invariant due to Freedmann--Kirby. 

\begin{definition}
A \emph{characteristic homology class}, with respect to the intersection form $\lambda_X$ of a $4$-manifold $X$, is an element $c \in H_2(X)$ such that $\lambda_X(c, x) \equiv \lambda_X(x,x) \pmod 2$ for every $x\in H_2(X)$. 
\end{definition}

First, we briefly recall how to define the \emph{Arf invariant}  $\Arf_X(c) \in \Z/2$ of a characteristic homology class $c \in H_2(W)$ in a \emph{smooth} $4$-manifold $W$. We choose a closed, oriented, 
smoothly embedded surface $\Sigma$ in $W$ representing~$c$.  In~\cite{Freedman-Kirby}, Freedman--Kirby defined a quadratic enhancement $q \colon H_1(\Sigma;\Z/2) \to \Z/2$ of the $\Z/2$-intersection form $\lambda_{\Sigma} \colon H_1(\Sigma;\Z/2) \times H_1(\Sigma;\Z/2) \to \Z/2$.  For a detailed definition of $q$, we refer to \cite[p.~121]{Matsumoto-Rochlin}, just above Lemma 1.1. We will not reproduce the full argument here, but we describe the outline.  One represents an element of $H_1(\Sigma;\Z/2)$ by a simple closed curve $\gamma$ in $\Sigma$, and chooses a generically immersed disc $C_\gamma \looparrowright W$ with boundary $\gamma$. Then $q([\gamma])$ is defined in terms of framing and intersection data derived from $C_\gamma$. 
With that we define \[\Arf_W(c) := \Arf(\Sigma) := \Arf(H_1(\Sigma;\Z/2),\lambda_{\Sigma},q) \in \Z/2,\]
where the last term is the algebraically defined $\Arf$ invariant of the quadratic form.

For a simply connected, topological $4$-manifold we can define the Arf invariant as follows.

\begin{definition}\label{defn:new-arf-defn}
    Let $X$ be an oriented, closed, simply connected, topological $4$-manifold and $c\in H_2(X)$ be a homology class. Then we define the \textit{Arf invariant} $\Arf_X(c)$ of $c$ as
    \begin{equation*}
        \Arf_X(c):=\Arf_W\big(\iota_*^{-1}(c)\big),
    \end{equation*}
    where $\iota\colon W\to X=W\cup_h C$ is the inclusion map of the handle part $W$ for some 
    decomposition of~$X$ into $W\cup_h C$ as in \cref{sec:input}.
\end{definition}

That this is well-defined, in particular that it does not depend on the choice of decomposition of $X$ into $W \cup_h C$, follows from the next theorem, essentially due to Freedman and Kirby~\cite{Freedman-Kirby}. We will use it later to compute the Kirby--Siebenmann invariant. 

\begin{theorem}[Freedman--Kirby]\label{theorem:Freedman-Kirby}
   Let $X$ be a closed, oriented, simply connected, topological $4$-manifold. 
    If $c$ is a characteristic homology class in $H_2(X)$ and $\sigma(X)$ is the signature of $X$, then
    \begin{equation*}
    \ks(X)=\Arf_X(c)+\frac{\lambda_X(c,c) - \sigma(X)}{8} \in \Z/2.
    \end{equation*}
    In particular, it follows that $\Arf_X(c)$ is well-defined.
\end{theorem}

Freedman--Kirby worked with smooth $4$-manifolds, so for them, the left-hand side was always zero. We will explain below how to deduce the statement we need from statements in the literature. It requires some work because the original references on Rochlin's theorem were written before Freedman's work on the disc embedding theorem, so did not consider topological $4$-manifolds. The statement we give was motivated by statements made in the introduction to~\cite{ConantSchneidermanTeichner_ksAndArf}. 

\begin{remark}\label{remark:arf-K}
Later we will relate $\Arf_X(c)$ to the Arf invariant of a knot $K$ determined by $L$ and~$c$. 
As explained by Matsumoto~\cite[p.~122]{Matsumoto-Rochlin}, we can compute the Arf invariant $\Arf(K)$ of a knot~$K$ in~$S^3$ (as defined in \cite{Robertello}) by taking a properly smoothly embedded, connected, orientable, compact surface $F$ in~$D^4$ with $\partial F=K$, and using the quadratic enhancement  $q \colon H_1(F;\Z/2) \to \Z/2$ to define $\Arf(F)$. Then Matsumoto showed that $\Arf(K) = \Arf(F)$.  
\end{remark}

\begin{proof}[Proof of \cref{theorem:Freedman-Kirby}]
We explain how to obtain the statement from results in the literature. Fix a decomposition $X = W \cup_Y C$ as in \cref{sec:input}, with $W$ smooth, $C$ contractible, and $Y$ an integer homology $3$-sphere. 
Since, as explained in~\cite{guide}, the Kirby--Siebenmann invariant is additive under glueing of $4$-manifolds along their boundaries, $\ks(X) = \ks(W) + \ks(C) = \ks(C)$, using that $W$ is smooth in the second equality. It is known that there exists a compact, spin, smooth $4$-manifold $N$ with $\partial N = Y=  \partial C$. Then $\ks(N \cup_Y C) = \ks(C)$ and $N \cup_Y C$ is spin, so \[\ks(C) = \ks(N \cup_Y C) \equiv \sigma(N \cup_Y C)/8 = \sigma(N)/8 \mod{2}.\]
The fact that $\ks(N \cup_Y C) \equiv \sigma(N \cup_Y C)/8$ follows from \cite[Proposition~10.2B]{FQ}.  
The last equality follows from Novikov additivity and $\sigma(C)=0$. 
    The quantity $\sigma(N)/8$ modulo 2 is by definition the \emph{Rochlin invariant} $\mu(Y)$ of $Y$. We deduce that $\ks(X) = \mu(Y)$. Let $\iota \colon W \to X$ be the inclusion. Recall that $\iota_* \colon H_2(W) \to H_2(X)$ is an isomorphism, see \cref{lemma-inclusion-W-iso-on-H2}.
    Then it follows from \cite[Equation~(2.4), p.~39]{Saveliev} that
    \[\mu(Y) = \Arf_W\big(\iota_*^{-1}(c)\big)+\frac{\lambda_W\big(\iota_*^{-1}(c),\iota_*^{-1}(c)\big) - \sigma(W)}{8}.\]
    As in the proof of \cref{lem:linking_equal_intersection}, we also have that $\iota_*$ is an isometry of intersection forms, so  $\lambda_X(c,c) = \lambda_W(\iota_*^{-1}(c),\iota_*^{-1}(c))$ and $\sigma(X) = \sigma(W)$. 
   Note that by \cref{defn:new-arf-defn}, we have $ \Arf_X(c)=\Arf_W(\iota_*^{-1}(c))$.
    We deduce that
   \[\ks(X)= \mu(Y) = \Arf_X(c)+\frac{\lambda_X(c,c) - \sigma(X)}{8},\] 
as desired. Since the Kirby--Siebenmann invariant is well-defined, it follows that the Arf invariant of a characteristic homology class in a simply connected, topological $4$-manifold is well-defined.
\end{proof}

To describe the algorithm to compute the Kirby--Siebenmann invariant we start with an algorithm to compute a characteristic vector. A \emph{characteristic vector} of an integral, symmetric, unimodular matrix $V$ is an integral vector~$c$ such that $c^T V x \equiv x^T Vx \pmod 2$ for every integral vector $x$. 
The following lemma shows that any integral, symmetric, unimodular matrix admits a characteristic vector that is algorithmically computable.

\begin{lemma}\label{lem:charVectAlgo}
    There exists an algorithm that
    \begin{itemize}
    \item takes as input an integral, symmetric, unimodular matrix $V$, and
    \item outputs a characteristic vector of $V$ with all entries either $0$ or $1$. 
    \end{itemize}
\end{lemma}

\begin{proof}
Let $V\in\GL_m(\Z)$ and let 
$V_2$ denote the reduction of $V$ modulo $2$. 
It is sufficient to find~$c$ with respect to $V_2$, as a vector in $(\Z/2)^m$, and then to consider its entries as integers. All subsequent computations are therefore done modulo $2$.

If all the diagonal entries of $V_2$ are zero (i.e.\ $V$ is even), then it is sufficient to take $c=0$. If there is a nonzero diagonal entry in $V_2$ (i.e.\ $V$ is odd), we can perform simultaneous row and column operations so that the matrix of the intersection form can be expressed as 
\[
\begin{pmatrix}
    1\\
    &V_2'
\end{pmatrix}
\]
for some $(m-1)\times (m-1)$ square matrix $V_2'$ over $\Z/2$ (cf.~\cref{lem:diagonal_form}). By iterating this process on $V_2'$, we can find a matrix $P\in\GL(m,\Z/2)$ such that 
\[ U:= P^TV_2P=\begin{pmatrix}
    1\\
    &\ddots\\
    &&1\\
    &&&E
\end{pmatrix} 
\]
for some even $(m-k)\times (m-k)$ matrix $E$ over $\Z/2$, i.e.\ all diagonal entries of $E$ are $0$. 
A characteristic vector of $U$ is given by $c^T:=(1,\dots,1,0,\dots,0)$ consisting of $k$-times $1$ followed by $(m-k)$-times $0$. To see this, we observe that $x_i = x^2_i$, and the fact that $E$ is even implies that $c^TUx = x^T U x$ for all $x \in (\Z/2)^m$. By undoing the basis change we see that $Pc$ is a characteristic vector for $V_2$ and thus also for $V$.
\end{proof}

\begin{proposition}\label{lem:Ks_algo}
   There exists an algorithm that 
   \begin{itemize}
       \item takes as input a framed, unimodular link $L$, and
       \item outputs the Kirby--Siebenmann invariant of $X_L$.
   \end{itemize}
\end{proposition}

\begin{proof}
    We write $X_L=W_L\cup_h C$. The link $L$ induces a preferred isomorphism $H_2(X_L) \cong \Z^m$. We use the algorithm from \cref{lem:computing_intersection_from} to compute an integer, symmetric, unimodular matrix $V$ representing the intersection form of $X_L$ in the above basis. Next, we compute a characteristic vector $c\in H_2(X_L) \cong \Z^m$ using the algorithm from \cref{lem:charVectAlgo}. As per that lemma, with respect to the standard basis of $\Z^m$, the entries of $c$ are either $0$ or $1$. 
    Let $L_c$ be the associated \emph{characteristic sublink} of $L$ (i.e.\ the components of $L$ indexed by the nonzero elements of $c$). Let $K_c$ in $S^3$ be a knot obtained by some choice of band sum combining the components of $L_c$ into a single component.  
    
    We claim that $\Arf_{X_L}(c) = \Arf(K_c)$. To see this, we construct a closed, oriented surface $\Sigma$ representing $c$ as follows. Consider the cores of the $2$-handles attached to the components of~$L_c$. Add the bands used in the construction of $K_c$. This gives a disc $D$ in $W_L$ with boundary $K_c$.  
    Choose an embedded, compact, oriented surface $F$ in the $0$-handle $D^4$ in $W_L$ such that $\partial F=K_c$. Take $\Sigma := F \cup D$ to obtain a closed surface embedded in $W_L$ representing $c$.   
    Then $H_1(F;\Z/2) \to H_1(\Sigma;\Z/2)$ is an isomorphism, and induces an isometry of quadratic forms (because the quadratic enhancements are geometrically computed in exactly the same way~\cite{Matsumoto-Rochlin}). Thus the Arf invariants of these surfaces coincide, and so as claimed 
    \[\Arf_{X_L}(c) = \Arf(\Sigma) =\Arf(F) = \Arf(K_c).\]
    The last equality follows by \cref{remark:arf-K}.
    Note it follows from the above claim and \cref{theorem:Freedman-Kirby} that
    \[
    \ks(X_L)=\Arf(K_c)+\frac{\lambda_{X_L}(c,c)-\sigma(X_L)}{8} = \Arf(K_c)+\frac{(c^T  V c) -\sigma(X_L)}{8},
    \]
    where in the first expression we consider $c \in H_2(X_L)$ and in the  last expression we consider $c \in \Z^m$. 

    The Arf invariant of the knot $K_c$ can be computed algorithmically from a diagram of $K_c$.
    For example Levine  \cite[p.~544]{Levine-Arf-Alexander} showed that for any knot $J$ with Alexander polnyomial $\Delta_J(t)$ we have $\Delta_J(-1)\equiv 1+4\operatorname{Arf}(J)\,\operatorname{mod} 8$. Alternatively, the Arf invariant can also be calculated from a Gauss diagram~\cite{PolyakViro}. 
    We explain how to compute the signature $\sigma(X_L)$ algorithmically in~\cref{section:intersection-form-algorithm}. 
    It is elementary to see that there exists an algorithm that makes an arbitrary choice of band sums to connect the components of $L$. In summary, we have shown that we can algorithmically compute $\ks(X_L)$ from a diagram of $L$, as desired. 
\end{proof}


\section{Comparing intersection forms}\label{section:intersection-form-algorithm}

 We recall from Section~\ref{sec:int_form} that the intersection form of a closed $4$-manifold can be represented by an integral, symmetric, unimodular matrix $V$. 
In this section, we will present an algorithm that decides whether or not two integral, symmetric, unimodular matrices are congruent. This result is probably known, but we could not find it in the literature, so we include a proof here. 

\begin{proposition}
\label{prop:intersectionform}
  There exists an algorithm that 
   \begin{itemize}
       \item takes as input two integral, symmetric, unimodular matrices $V$ and $V'$, and 
       \item outputs whether or not these two matrices are congruent over $\Z$.
   \end{itemize} 
\end{proposition}

In general, an integral, symmetric, unimodular matrix $V$ is not congruent to a diagonal matrix over $\Z$. However, the following lemma shows that over $\Q$ the matrix $V$ is always congruent to a diagonal matrix.

\begin{lemma}\label{lem:diagonal_form}
    There exists an algorithm that
    \begin{itemize}
        \item takes as input an integral, symmetric, unimodular matrix $V$, and
        \item outputs an integral matrix $P$ with determinant $\det(P)\in\Z\setminus\{0\}$ such that $P^TVP$ is an integral diagonal matrix. 
    \end{itemize}
\end{lemma}
\begin{proof}
    Write $V = (v_{ij})_{1\leq i, j\leq m}$. First we assume $v_{11}\neq 0$. For $k=2,\dots,m$ multiply the $k$-th column by $v_{11}$ and then subtract $v_{1k}$ times the first column from the $k$-th column. Perform the analogous operation on rows. This corresponds to replacing~$V$ with~$P_{1,k}^TVP_{1,k}$ where
    \[
    P_{1,k}=\begin{pmatrix}
        1 &&& -v_{1k} \\&\ddots\\&&1\\&&&v_{11}\\&&&&1\\&&&&&\ddots\\&&&&&&1 
    \end{pmatrix}.
    \]
    After replacing~$V$ with~$P_1^TVP_1$, for $P_1= P_{1,2} \cdots P_{1,m}$,
    this results in all entries of the first row and first column except for $v_{11}$ being zero, i.e.\ the matrix is of the form
    \[
    \begin{pmatrix}
        v_{11}\\&V'
    \end{pmatrix}.
    \]           
    If $v_{11}=0$, the unimodularity of~$V$ implies that there exists $k\in\{2,\dots,m\}$
    with $v_{1k}=v_{k1}\neq0$. Add the $k$-th row to the first row and the $k$-th column to the first column, that is, replace~$V$ with~$Q_{1,k}^TVQ_{1,k}$, where $Q_{1,k}$~is obtained from the identity matrix by replacing the $(k,1)$-entry with a~$1$. The $(1,1)$-entry of the resulting matrix is no longer zero, so we can continue as in the $v_{11}\neq0$ case. In this case $P_1 = Q_{1,k} \cdot P_{1,2} \cdots  P_{1,m}$. 

    Then we repeat the above process to simultaneously simplify the second row and the second column, and so on. This process stops after finitely many steps and yields an integral matrix $P=P_1\cdots P_m$
    with non-vanishing determinant such that $P^TVP$ is an integral diagonal matrix.
\end{proof}

A direct corollary is that the signature and the definiteness status are computable.

\begin{corollary}\label{cor:alg-for-signature}
    There exists an algorithm that
    \begin{itemize}
        \item takes as input an integral, symmetric, unimodular matrix $V$, and 
        \item outputs the signature of $V$ and whether $V$ is positive definite, negative definite, or indefinite.
    \end{itemize}
\end{corollary}

\begin{proof}
    By \cref{lem:diagonal_form} we can algorithmically find a diagonal matrix $D$ that is congruent to $V$ over~$\Q$. In particular, $D$ has the same definiteness as $V$. Hence, $V$ is positive/negative definite if all diagonal entries of $D$ are positive/negative, respectively, and otherwise, $V$ is indefinite. 
    Similarly, the signature of $V$ equals the signature of $D$ and is given by the number of positive entries in $D$ minus the number of negative entries in $D$.
\end{proof}

As a corollary, we obtain the algorithmic classification of indefinite forms.

\begin{corollary}\label{cor:indefinitecase}
There exists an algorithm that 
\begin{itemize}
    \item takes as input two integral, symmetric, unimodular, indefinite matrices $V$ and $V'$, and
    \item outputs whether or not $V$ and $V'$ are congruent over $\Z$.
\end{itemize}
\end{corollary}

\begin{proof}
    Since $V$ and $V^\prime$ represent indefinite forms, they are congruent over $\Z$ if and only if they have the same parity, rank, and signature~\cite[Theorem II.5.3]{MH1973}, cf.\ \cite[Chapter IV, Theorem~8]{Serre}. Parity, rank, and signature can be computed algorithmically. Indeed, the signature was discussed in \cref{cor:alg-for-signature}. For the parity, note that a matrix $V$ is odd if and only if it has some odd diagonal entry, which can be checked algorithmically. 
\end{proof}
    
It 
remains to discuss the definite case.
We will focus here only on the positive definite case, as the proof for the negative definite case can be obtained by replacing $V$ with $-V$. 

\begin{lemma}\label{lem:definitecase}
    There exists an algorithm that
    \begin{itemize}
        \item takes as input two integral, positive definite, symmetric, unimodular matrices $V$ and $V'$, and
        \item outputs whether or not $V$ is congruent to~$V'$ over $\Z$.
    \end{itemize}
\end{lemma}

The strategy to prove \cref{lem:definitecase} is to construct a finite set of lattice points that contains the image of the standard basis of $\Z^m$ under any isometry between the two given forms. This reduces the search for an isometry to a finite list of matrices. We start with a lemma from linear algebra for which we introduce the following notation. We write $|x|$ for the Euclidean norm (i.e.\ $\ell^2$-norm) of a vector $x$. For a matrix~$A$ with real coefficients we write $\Vert A\Vert_2$ and $\Vert A\Vert_1$, respectively, for the operator norms with respect to the $\ell^2$ and $\ell^1$ norms. 
We recall that for $A$ an $(m\times m)$ matrix,
$\Vert A\Vert_1 = \max_{j} \sum_{i=1}^m |a_{ij}|$ and $\Vert A \Vert_2$ is the square root of the largest eigenvalue of $A^TA$. Note that for a symmetric matrix $A$, it follows $\Vert A\Vert_2=\max\{|\lambda|\,:\, \lambda \,\textrm{ is an eigenvalue of $A$}\}$.

 \begin{lemma}\label{lem:upperbound}
   Let $V$ be an integral, unimodular, symmetric, positive-definite matrix. Then for every integral vector $x$, we have
   $$|x|\leq \lVert V^{-1}\rVert_1 \, (x^TVx).$$
 \end{lemma}

\begin{proof}
First, we observe that if $\lambda$ is an eigenvalue of a matrix $A$ with eigenvector $x$, then
\begin{equation*}
    |\lambda|\cdot \Vert x\Vert_1=\Vert \lambda x\Vert_1=\Vert A x\Vert_1\leq \Vert A\Vert_1 \cdot \Vert x\Vert_1,
\end{equation*}
and thus $|\lambda|\leq \Vert A\Vert_1.$
So if $A$ is symmetric, then 
\begin{equation}\label{eqn:2smaller1}
    \Vert A\Vert_2\leq \Vert A\Vert_1.
\end{equation}
Since $V$ is positive definite, Sylvester's theorem implies that there exists an invertible matrix $P$ with real coefficients such that $V=P^TP$.
For $P$ we estimate
\begin{align}\label{eqn:basic-properties}
    \lVert P^{-1}\rVert_2  \, &= \, \lVert (P^{-1})^T \rVert_2
    \,=\, \sqrt{\lVert P^{-1} (P^{-1})^T \rVert_2} \, = \, \sqrt{\lVert (P^TP)^{-1} \rVert_2}\\
    \,&=\,\sqrt{\lVert V^{-1}\rVert_2} \,\leq\,\sqrt{\lVert V^{-1} \rVert_1}\,\leq\,\lVert V^{-1} \rVert_1. \nonumber
\end{align}
    Here the first equality uses that $\|A^T\|_2 = \|A\|_2$ and the second equality is the $C^*$-identity $\|A\|^2_2 = \|A^T A\|_2$, both applied with $A = (P^{-1})^T$.  
    The first inequality is Equation~\eqref{eqn:2smaller1}, which we can apply with $A=V^{-1}$ since $V^{-1}$ is symmetric because $V$ is. The last inequality holds because $\lVert V^{-1} \rVert_1$ is a non-negative integer.
Next, we observe that  
\begin{equation}\label{eqn:some-more-very-basic-stuff}
|Px|=\sqrt{(Px)^TPx}=\sqrt{x^TVx}\leq x^TVx,
\end{equation}
where the last inequality holds because $x^TVx$ is a non-negative integer.
Combining this, we estimate that
\begin{equation*}
|x|  \,=\,| P^{-1} P x| \,\leq\,  \Vert P^{-1}\Vert_2 \cdot | Px| \, \leq\, \Vert V^{-1}\Vert_1 \, (x^TVx),  
\end{equation*}
as desired. 
Here the first inequality uses a property of any operator norm $\|-\|_{\operatorname{op}}$ that $|Ay| \leq \|A\|_{\operatorname{op}}\cdot |y|$. The second inequality combines Equations~\eqref{eqn:basic-properties} and~\eqref{eqn:some-more-very-basic-stuff}.  
\end{proof}

\begin{proof}[Proof of \cref{lem:definitecase}]
    Let $V, V'\in\GL(m,\Z)$ be positive definite, symmetric matrices. For the standard basis $e_1,\ldots,e_m$ of $\Z^m$, we define the integer
    $$R:= \max \big\{e_i^T V' e_i \,:\, 1\le i \le m\big\}.$$
    If $f\colon(\Z^m, V')\to (\Z^m, V)$ is an isometry, then
    $f(e_i)^TVf(e_i)=e_i^TV'e_i\leq R$,
    and thus $f$ maps the basis vectors~$\{e_1, \ldots, e_m\}$ into the set
    $$F:= \big\{x\in \Z^m \,:\, x^T V x \le R\big\}.$$
    Since $V$ is positive definite, Lemma~\ref{lem:upperbound} implies that the set~$F$ is contained in the finite ball
    \[
    B:=\big\{x \in\Z^m\,:\, |x|\leq \Vert V^{-1}\Vert_1 R\big\}.
    \]
    It follows that every isometry $(\Z^m,V')\to (\Z^m,V)$ is given by an integral, unimodular matrix whose columns are lattice points contained in $B$. 

    The algorithm to check if $V$ and $V'$ are congruent thus works as follows. Since $V$ and $V'$ are integral matrices, we can compute $\Vert V^{-1}\Vert_1 R$. Then we can build all (finitely many) $(m\times m)$-matrices with column vectors in $B$. If there exists such a matrix~$A$ that is unimodular with $V' = A^TVA$, we conclude that the forms $V'$ and~$V$ are isometric; otherwise they are not.
    \end{proof}

 Combining the various algorithms from this section, we obtain the claimed main result of this section.

 \begin{proof}[Proof of \cref{prop:intersectionform}]
    Let $V$ and $V$ be integer, symmetric, unimodular matrices. First, we use \cref{cor:alg-for-signature} to determine the definiteness status of $V$ and $V'$. If $V$ and $V'$ have different definiteness statuses the matrices are not congruent.

    If both matrices are indefinite, we use \cref{cor:indefinitecase} to determine if $V$ and $V'$ are congruent. If both matrices are positive definite we use \cref{lem:definitecase} to determine if $V$ and $V'$ are congruent. If both matrices are negative definite, we observe that $-V$ and $-V'$ are positive definite and congruent if and only if $V$ and $V'$ are congruent, and thus we can reduce this case to the positive definite case.    
 \end{proof}


\section{The homeomorphism problem for simply connected 4-manifolds}\label{section:proof-main-thm}

With the algorithms to compute the Kirby--Siebenmann invariant and to compare intersection forms, we can now deduce our main result. We recall, that by \cref{thm:topKirby} a framed, unimodular link determines the oriented homeomorphism type of a closed, oriented, simply connected, topological $4$-manifolds $X_L$.

\begin{theorem}\label{thm:main}
    There exists an algorithm that 
    \begin{itemize}
        \item takes as input two framed, unimodular links $L$ and $L'$, and
        \item outputs whether or not $X_L$ and $X_{L'}$ are orientation-preserving homeomorphic.
    \end{itemize} 
\end{theorem}

\begin{proof}
    Let $L$ and $L'$ be two framed, unimodular links.  
    By Freedman's Theorem~\ref{thm:freedman}
    the two manifolds $X_L$ and $X_{L'}$ are orientation-preserving homeomorphic if and only if they share the same Kirby--Siebenmann invariant and their intersection forms are isometric.

    By \cref{lem:computing_intersection_from} we can compute matrices $V$ and $V'$ representing the intersection forms of $X_L$ and $X_{L'}$ and using \cref{prop:intersectionform} we can decide whether or not these two forms are isometric. Furthermore, we can compute the Kirby--Siebenmann invariants of $X_L$ and $X_{L'}$ from the framed, unimodular links, see~\cref{lem:Ks_algo}. 
    
    If both, the Kirby--Siebenmann invariants agree and the intersection forms are isometric, the manifolds $X_L$ and $X_{L'}$ are orientation-preserving homeomorphic. Otherwise, they are not orien\-ta\-tion-preserving homeomorphic.
\end{proof}

\begin{remark}
    We briefly discuss the runtime of the algorithm from \cref{thm:main}. It seems that the bottleneck of our algorithm is in the algorithm from \cref{lem:definitecase}, which compares two definite intersection forms. There we are enumerating all matrices with columns from a finite set whose cardinality grows in the input size, and thus this algorithm has factorial runtime. On the other hand, most other parts of the algorithm from \cref{thm:main} actually run in polynomial time. 
    
    We also observe that if the input manifolds both admit smooth structures, then the runtime of the above algorithm can be improved drastically. Indeed, if $X_1$ and $X_2$ both admit smooth structures, then their Kirby--Siebenmann invariants vanish and thus we do not need to run through the algorithm from \cref{lem:Ks_algo}. Moreover, Donaldson's theorem~\cite{Donaldson} implies that if the intersection form of a closed, oriented, smooth manifold is positive definite then in some basis it is given by the identity matrix. Thus in the smooth case, it follows together with \cref{cor:indefinitecase} that two intersection forms are isometric if and only if they have the same rank, parity, and signature. 
\end{remark}


\section{An algorithm for stably classifying smooth 4-manifolds}\label{sec:stable_algo}

The aim of this section is to provide companion algorithms to the rest of the article which classify certain classes of closed, smooth $4$-manifolds up to stable diffeomorphism. The stable classification is coarser, and hence easier to compute. This means that we can leave the realm of simply connected $4$-manifolds. In this section, all manifolds are assumed to be connected, closed, oriented, and smooth. We will consider \emph{pointed} manifolds without further comment.  

\subsection{Input}\label{subsec:Input}

Before stating our main theorems we briefly explain how we can input smooth $4$-manifolds with non-trivial fundamental groups into algorithms.
The setup in the previous sections of the paper allows us to input closed, simply connected, topological $4$-manifolds into algorithms. This relies on the theorem saying that every closed, simply connected, topological $4$-manifold can be written as a smooth piece union a contractible topological piece (see \cref{thm:topKirby}). This is not known to hold for closed, topological $4$-manifolds with other fundamental groups, but is known to hold stably, i.e.\ every topological 4-manifold is stably homeomorphic to a smooth manifold union a contractible piece.\footnote{If the manifold $X$ is stably smoothable then the statement is clear.  If it is not consider $X\# \ol{E_8}\# E_8$, which is stably homeomorphic to $X$. Now $E_8$ is smooth away from a contractible piece $C$, and so we can stably smooth $X\# \ol{E_8} \# E_8$ away from $C$.}  Likely one could combine the work of the previous sections and this one to obtain results about stable homeomorphism for general topological 4-manifolds, but this would require us to input smooth data that is only abstractly known to emerge from stabilisations.  Since this is somewhat unsatisfactory, we will instead only consider smooth 4-manifolds as inputs.

For our purpose a \textit{triangulated} $4$-manifold is a compact simplicial complex whose underlying topological space is homeomorphic to a closed $4$-manifold. Also, we need to work with \emph{oriented} manifolds, so our triangulations will come with oriented $4$-simplices, i.e.\ for each $4$-simplex we are given an equivalence class of ordering of its vertices, where two orderings are considered the same if one can be obtained from the other by an even number of swaps. The orientations need to fit together on the $3$-simplices. 

In fact, a triangulation determines a unique smooth structure on the underlying $4$-manifold.

\begin{lemma}
Let $\mathcal{K}$ be an oriented triangulation of a $4$-manifold $X$. Then $\mathcal{K}$ induces a preferred smooth structure on $X$, that is unique up to orientation-preserving diffeomorphism.
\end{lemma}

\begin{proof}
    From the resolution of the Poincar\'{e} conjecture in dimension three~\cite{perelman1,perelman2} it follows that in
     dimension four, a triangulated manifold can be given an essentially unique PL-structure. We refer to \cite[Chapter 3]{guide} for more details. In dimension four a PL-manifold can be given an essentially unique smooth structure (due to \cite{Hirsch_Mazur, munkres_1, munkres_2, cerf}, cf.\ \cite[Theorem 8.3B]{FQ}). 
\end{proof}

From an algorithmic point of view triangulations and Kirby diagrams are equivalent as the following lemma shows.

\begin{lemma}\label{rem:equivalence_Kirby_triangulation}
   There exist an algorithm that 
   \begin{itemize}
       \item takes as input an oriented triangulation $($Kirby diagram$)$ of a closed, 
       oriented, connected,  
       smooth $4$-manifold $X$, and
       \item outputs a Kirby diagram $($triangulation$)$ of $X$.
   \end{itemize}
\end{lemma}

\begin{proof}
    It is straightforward to describe algorithms to create a triangulation out of a Kirby diagram. (This is even practically implemented in Regina \cite{burke2024practicalsoftware}.) On the other hand, since any two triangulations of the same smooth manifold are related by Pachner moves and the set of Kirby diagrams is countable, there also exists an (impractical) algorithm to create a Kirby diagram out of a triangulation.
\end{proof}

\subsection{Algorithms for stable diffeomorphism}

We are now ready to state our main results of this section.

\begin{theorem}\label{thm:mainAlmostSPinOrSpin}
    There exists an algorithm that 
    \begin{itemize}
        \item takes as input oriented triangulations of closed, oriented, smooth $4$-manifolds $X_1$ and $X_2$ such that either 
        \begin{enumerate}
            \item\label{item:thm-8-3-i} their fundamental groups are both isomorphic to the infinite cyclic group,
            \item\label{item:thm-8-3-ii} or $X_1$ has a finite fundamental group with a cyclic Sylow $2$-subgroup and $X_2$ has a finite fundamental group, and
        \end{enumerate}
        \item outputs whether or not $X_1$ and $X_2$ are orientation-preserving stably diffeomorphic. 
    \end{itemize}
\end{theorem}

\begin{theorem}\label{thm:mainNon-Spin}
    There exists an algorithm that 
    \begin{itemize}
        \item takes as input oriented triangulations of closed, oriented, smooth $4$-manifolds $X_1$ and $X_2$, such that both universal covers $\widetilde{X}_1$ and $\widetilde{X}_2$ are not spin and such that
        \begin{enumerate}
            \item\label{item:thm-8-4-i} their fundamental groups are isomorphic and of homological dimension $\leq 3$,
            \item\label{item:thm-8-4-ii} or their fundamental groups are both finite, and
        \end{enumerate}
        \item outputs whether or not $X_1$ and $X_2$ are orientation-preserving stably diffeomorphic. 
    \end{itemize} 
\end{theorem} 

We note that in \cref{thm:mainAlmostSPinOrSpin,thm:mainNon-Spin} above, in the case of finite $\pi_1$ (i.e.\ in case (2) in each theorem) we do not need 
to assume that the manifolds have \emph{isomorphic} fundamental groups because in this 
case the word problem is decidable (see \cref{thm:WordProblemFinitePi}). 

We recall the basics of modified surgery in \cref{subsec:Input}. The main result of this theory (for our interests) is \cref{thm:stableclassif}, which states that the set of stable diffeomorphism classes of $4$-manifolds with a fixed 1-type $\xi$ (\cref{def:normal1types}) is in one-to-one correspondence with $\Omega^{\xi}_4/\operatorname{Out}(\xi)$, a quotient of a certain bordism group. If either of the assumptions of \cref{thm:mainAlmostSPinOrSpin,thm:mainNon-Spin} is satisfied, we will prove that the following map given by the signature $\sigma$, and the primary invariant $\pri$ (given by evaluating the fundamental class)
\[\Omega_4^{\xi}\xrightarrow{\sigma+\pri} \Z\oplus H_4(B\pi)\]
is injective. We show that these invariants are both computable from the input data. 

\subsection{Combinatorial modified surgery}

There is some additional setup needed that we introduce next.  Let $BSO$ denote the classifying space of the direct limit of the special orthogonal groups $SO(n)\hookrightarrow SO(n+1)$.  Modified surgery, developed by Kreck \cite{Kreck}, stably classifies manifolds by considering certain approximation spaces, denoted by $B$, that admit a highly coconnected map to $BSO$. In the end, $B$ will be an approximation of a given manifold $M$, in the sense that $B$ will be a Moore--Postnikov approximation for the stable normal bundle of $M$.  We make this all precise now.

\begin{definition}
    Let $\xi\colon B\to \BSO$ be a map from some CW-complex $B$ with finite $k$-skeleton for all $k$. We say that $\xi$ is a \emph{universal fibration} if it is a 2-coconnected fibration. In particular, this means that $\xi_*\colon \pi_k(B)\to \pi_k(\BSO)$ is injective for $k=2$ and an isomorphism for $k\geq 3$.
\end{definition}

\begin{definition}\label{def:normal1types}
    Let $\xi\colon B\to \BSO$ be a universal fibration and let $\nu_X\colon X\to \BSO$ be the stable normal bundle for a closed smooth $4$-manifold $X$.  We say that $\xi$ is a \emph{normal 1-type} for $X$ if there exists a 2-connected lift $\ol{\nu}_X$ such that the diagram
    \[
    \begin{tikzcd}
        & B \arrow[d,"\xi"] \\
        X \arrow[ur,"\ol{\nu}_X"] \arrow[r,"\nu_X"'] & \BSO
    \end{tikzcd}
    \]
    commutes. We call a  choice of $\ol{\nu}_X$ a \emph{normal 1-smoothing}. We call a pair $(X,\ol{\nu}_X)$ consisting of a manifold and a normal $1$-smoothing a $\xi$-manifold. A $\xi$-diffeomorphism of two $\xi$-manifolds is a diffeomorphism of the underlying manifold that commutes with the normal $1$-smoothing.
\end{definition}

\begin{theorem}[Kreck \cite{Kreck}]\label{thm:stableclassif}
    Let $\xi$ be a universal fibration. The stable $\xi$-diffeomorphism classes of $\xi$-manifolds are in one-to-one correspondence with elements of the bordism group $\Omega_4(\xi)$.\qed
\end{theorem}

Since our manifolds are oriented we get that $w_2(X)$, the second Stiefel--Whitney class of $TX$, agrees with the second Stiefel--Whitney class of the stable normal bundle of $X$. We will work with the tangential classes since it is more convenient. 

\begin{theorem}[Various normal 1-types, Teichner \cite{teichnerthesis}]
    Let $\pi$ be a finitely presented group, let $X$ be an orientable, smooth $4$-manifold with $\pi_1(X)$ isomorphic to $\pi$ and let $\widetilde{X}$ denote the universal cover of $X$. Then the possible normal $1$-types for $X$ fit into the following three cases.
    \begin{itemize}
        \item[Type I]  \emph{(totally non-spin).}  If $w_2(\widetilde{X})\neq0$, the normal 1-type of $X$ is equivalent to 
        \[
        \begin{tikzcd}
            \BSO \times B\pi \arrow[r,"\xi"] &
            \BSO,
        \end{tikzcd}
        \]
        where the universal fibration $\xi$ is given by forgetting $B\pi$. A normal $1$-smoothing is given by a choice of orientation of $X$ and an isomorphism $\pi_1(X)\to \pi$.
        \item[Type II] \emph{(spin).} If $w_2(X)=0$, the normal $1$-type is equivalent to 
        \[
        \begin{tikzcd}
            \BSpin \times B\pi \arrow[r,"\xi"] &
            \BSO,
        \end{tikzcd}
        \]
        where the universal fibration $\xi$ is induced by the natural map from $\Spin \to \SO$ after first forgetting $B\pi$.  A normal $1$-smoothing is given by a choice of spin structure on $X$ and an isomorphism $\pi_1(X)\to \pi$.
        \item[Type III] \emph{(almost spin).}  If $w_2(X)\neq 0$, but $w_2(\widetilde{X})=0$, there is an element $w_2^\pi\neq 0 \in H^2(B\pi;\Z/2)$, such that $w_2(X)$ is pulled back from $w_2^\pi$ along a 2-connected map $X \to B\pi$.
        The normal $1$-type $\xi$ is equivalent to the map given by the homotopy pullback
        \[
        \begin{tikzcd}
            B \arrow[d, "\xi"] \arrow[r] \arrow[dr, phantom, "\scalebox{1.5}{\color{black}$\lrcorner$}" , very near start, color=black] & B\pi \arrow[d,"w_2^\pi"] \\
            \BSO \arrow[r," w_2"'] & K(\Z/2,2)
        \end{tikzcd}
        \]
        where $w_2$ denotes the universal map for the second Stiefel--Whitney class.  A $1$-smoothing is given by an isomorphism $\pi_1(X)\to \pi$ such that $w_2(X)$ pulls back from $w_2^\pi$.\qed
    \end{itemize}
\end{theorem}

It will be useful in our algorithms to have a `certificate' which states that the fundamental groups of the two manifolds in question are isomorphic.

\begin{definition}\label{def:proofPi1Isomorphic}
    Let $\pi$ be either a finite group or $\Z$, let $X$ be a triangulated manifold, and let $\varphi\colon\pi\to\pi_1(X)$ be an isomorphism.
  The following constitutes \emph{data for the isomorphism $\varphi\colon\pi\to \pi_1(X)$} for the finite case and the infinite cyclic case respectively. 
     \begin{enumerate}
        \item A map that assigns to every element $g$ in $\pi$ a based loop $\gamma_g$ in the $1$-skeleton of $X$, given by concatenation of edges, such that $\gamma_g$ is a representative of $\varphi(g)$.
        \item A based loop $\gamma$ in the $1$-skeleton of $X$, given by concatenating edges, such that $\gamma$ is a representative of $\varphi(1)$.
    \end{enumerate}
\end{definition}

\begin{definition}[1-type data]
    Let $\xi$ be a universal fibration, $\pi$ be a finite group, and let $X$ be a closed, smooth, oriented $4$-dimensional $\xi$-manifold with $\pi_1(X)$ isomorphic to $\pi$. We say that a \emph{1-type datum} for $X$ is a triple 
     \[(T_{\pi},(B\pi)^{(5)},w_2^{\pi}),\] 
     where $T_\pi$ is a multiplication table for the finite group $\pi$, $(B\pi)^{(5)}$ is the 5-skeleton of some triangulation of some $B\pi$, and $w_2^{\pi}$ is an element of the set  $H^2(B\pi^{(5)};\Z/2)\cup \{\infty\}$, where $\{\infty\}$ is a set with one element disjoint from the cohomology group $H^2(B\pi^{(5)};\Z/2)$. 
 We require that 
         $w_2^{\pi}=\infty$ if $w_2(\wt{X})\neq 0$,
         while if $w_2(\wt{X})=0$, we require that there exists a $\pi_1$-isomorphism $c\colon X\to B\pi^{(5)}$ such that $c^*(w_2^{\pi}) = w_2(X) \in H^2(X;\Z/2)$.
\end{definition}

\begin{definition}[Reduced 1-smoothing data] 

    For a closed, smooth, oriented $4$-dimensional $\xi$-manifold $X$, with triangulation $\mathcal{K}$, together with 1-type data $(T_{\pi}, (B\pi)^{(5)},w_2^{\pi})$, let $\iota\colon X\to B\pi_1(X)^{(5)}$ be the Postnikov truncation.  Then the \emph{reduced smoothing data} consists of a finite subset $A$ of the set of all maps 
    \[\{c\colon B\pi_1(X)^{(5)}\to B\pi^{(5)} \,:\,
w_2(X)=(c\circ\iota)^*(w_2^{\pi})\ \text{if}\ w_2^\pi\neq \infty\},\] 
such that
\begin{enumerate}
    \item all maps in $A$ when pre-composed with $\iota$ are simplicial with respect to the two-fold  barycentric subdivision of $\mathcal{K}$; 
    \item for every isomorphism $f\colon\pi_1(X)\to \pi$ with the property $\iota^*(Bf)^{*}(w_2^{\pi})=w_2(X)$ there is a map $c_f \colon B\pi_1(X)^{(5)} \to B\pi^{(5)}$ in $A$ with $(c_f)_* = f \colon \pi_1(X)\to \pi$.  
\end{enumerate} 
\end{definition}

We note that the reduced 1-smoothing data does \emph{not} give complete information about 1-smoothings, but the data is sufficient for our purposes.  The next lemma shows that reduced smoothing data can always be found. 

\begin{lemma}\label{lem:simplicialMap}
    Let $X$ be a triangulated $4$-dimensional manifold, let $\pi$ be a finite group and let $f\colon \pi_1(X)\to \pi$ be a homomorphism. Then there is a simplicial map between the two-fold barycentric subdivision of $X$ and the two-fold barycentric subdivision of the geometric realisation of the bar construction $B\pi$, which, on the level of fundamental groups, agrees with $f$.  Further, we can factor this simplicial map as a composition of simplicial maps $X\to B\pi_1(X)\to B\pi$ such that the first map induces the identity map and the second map induces $f$ on the level of fundamental groups.
\end{lemma}

\begin{proof}
    We will define a simplicial map between the given triangulation of $X$ and the $\Delta$-complex $B\pi$ (whose two-fold barycentric subdivision is a simplicial complex). Recall that we can construct $B\pi$ as the geometric realisation of the semisimplicial set with $k$-simplices are given by $(B\pi)_k = \pi^k$, where we denote by $|g_1|g_2|\dots|g_k|$ the $k$-simplex corresponding to $(g_1,g_2,\dots,g_k)\in \pi^k$, and where the face maps are
    \begin{align*}
f_0(|g_1|g_2|\dots|g_k|)&=|g_2|\dots|g_k|,\\f_i(|g_1|g_2|\dots|g_k|)&=|g_1|g_2|\dots|g_i\cdot g_{i+1}|\dots|g_k|\quad 1\leq i< k,\\f_{k}(|g_1|g_2|\cdot|g_k|)&=|g_1|g_2|\dots|g_{k-1}|.
    \end{align*}
    Now construct a map of $\Delta$-complexes from $X$ to $B\pi$ as follows. 
    \begin{itemize}
    \item Map the 0-skeleton of $X$ to the unique 0-simplex of $B\pi$.
    \item On the 1-skeleton of $X$, choose a maximal tree $T\subset X_{(1)}$ and send $T$ to the only $0$-simplex in $B\pi$. For any other 1-simplex $e$ of $X$, pick any based loop in $T\cup e$ which goes through $e$ exactly once (in the correct direction). This determines an element $g\in \pi_1(X)$. We send $e$ to the 1-simplex $|f(g)|$ of $B\pi$.
    \item Since $B\pi$ has vanishing higher homotopy groups, there always exists an extension (unique up to homotopy) of the above map $X_{(1)} \rightarrow B\pi$ to the whole $X$. Here is one explicit way to do this: for any $k$-simplex in $X$, if the edges $\{0\rightarrow 1\}, \{1\rightarrow 2\}, \dots, \{k-1 \rightarrow k \}$ are sent to $|g_1|, \dots, |g_k|$, then we send this $k$-simplex to $|g_1|\dots|g_k|$. The compatibility with face maps is easy to check.
    \end{itemize}  
    To turn this map of $\Delta$-complexes into a map of simplicial complexes, we simply carry out barycentric subdivision twice on both $X$ and $B\pi$, recalling that a $\Delta$-complex subdivided twice yields a simplicial complex.

    To achieve the factorisation at the end of the lemma, apply the previous construction in the special case $\pi=\pi_1(X)$ and $f=\Id$.  Then the above procedure also gives a method for producing a simplicial map from $B\pi_1(X)\to B\pi$ which induces the map $f$ on fundamental groups (in fact it is easier since one does not need to consider a maximal tree).  Since these two constructions give the same map on the level of fundamental groups, this gives a factorisation of the previous simplicial map.
\end{proof}

\subsection{Group theoretic considerations}

We will now show how to remove the need to add in a certificate that the fundamental groups of our given manifolds are isomorphic to the input in \cref{thm:mainAlmostSPinOrSpin} and \cref{thm:mainNon-Spin}, in the case that the fundamental groups are finite or are abstractly isomorphic.  In particular, we will show that one can produce such a certificate in these cases.  

Word problems for general groups are undecidable, but for finitely presented finite groups this is not the case.

\begin{definition}
    Let $\mathscr{S}$ be a 
    class of groups. We say that the \emph{word problem} is solvable for $\mathscr{S}$ if for every presentation of a group in $\mathscr{S}$ there exists an algorithm which determines for every word in terms of generators or their inverses 
    whether it is the identity in the given group.
\end{definition}

The word problem is in general undecidable, but there are decidability results for some classes of groups.
\begin{theorem}[\cite{McKinsey}]\label{thm:WordProblemFinitePi}
The word problem for finitely presented residually finite groups is solvable.\qed
\end{theorem}

Finite groups are a special case of residually finite groups. The corollary that we need is the following. 

\begin{corollary}\label{cor:CompareGroups}
    There exists an algorithm that 
    \begin{itemize}
        \item takes as input two finite presentations $\mathcal{P}_1,\mathcal{P}_2$ of finite groups $G_1,G_2$, and 
        \item outputs an isomorphism from $G_1$ to $G_2$ or certifies that there exists no such isomorphism.
    \end{itemize} 
\end{corollary}

\begin{proof}
For each presentation we generate a multiplication table for the corresponding group. We begin creating the table by adding all generators as rows/columns for the table. Then iterate the following procedure. For each product of two elements in the table, check if it is already in the table by using the algorithm in \cref{thm:WordProblemFinitePi} which solves the word problem.  If it is not, then this determines an entry and a new row/column.  If it is, then this determines a new entry. Continue until there are no more products which do not occur in the table, which will happen in finite time since we knew the groups were both finite.  Since we started by adding all of the generators, this table gives a multiplication table for the group corresponding to the presentation.

For the presentations $\mathcal{P}_1$ and $\mathcal{P}_2$ we compare the resulting tables by trying to find an isomorphism between them. This is done by iterating all bijections and checking if the given bijection gives an isomorphism.
\end{proof}

\begin{remark}
    One could avoid the use of \cref{thm:WordProblemFinitePi} in the proof of \cref{cor:CompareGroups} by using the fact that finite groups are classified and hence have an enumeration.  Since the classification of finite groups is a much deeper result, we instead use the more classical fact that the word problem is solvable.
\end{remark}

For handling the infinite cyclic case, we have a similar corollary.

\begin{corollary}\label{cor:CompareGroups_infcyclic}
There exists an algorithm that 
    \begin{itemize}
        \item takes as input two finite presentations $\mathcal{P}_1,\mathcal{P}_2$ which are abstractly known to be isomorphic, and 
        \item outputs an isomorphism between these two groups.
    \end{itemize} 
\end{corollary}

\begin{proof}
    Assume $\mathcal{P}_1$ and $\mathcal{P}_2$ have $n_1$ and $n_2$ generators, respectively.  Enumerate all homomorphisms between the free groups $F_{n_1}$ and $F_{n_2}$.  For each homomorphism, we can determine if this descends to a homomorphism $G_{\mathcal{P}_1}\to G_{\mathcal{P}_2}$, by checking to see if each relation is sent to the trivial element. Such an algorithm terminates in finite time if the answer is positive, and runs indefinitely if the answer is negative.  By running this algorithm diagonally, this allows us to enumerate all homomorphisms $G_{\mathcal{P}_1}\to G_{\mathcal{P}_2}$.  Similarly enumerate all homomorphisms $G_{\mathcal{P}_1}\to G_{\mathcal{P}_2}$.  Now check whether each possible composition of $f$ and $g$ where $f\colon G_{\mathcal{P}_1}\to G_{\mathcal{P}_2}$ and $g\colon G_{\mathcal{P}_1}\to G_{\mathcal{P}_2}$ is the identity map (again in a diagonal manner).  Eventually this process will end, since we abstractly know that the groups are isomorphic.
\end{proof}

Applying these corollaries to a presentation of the fundamental group of a triangulation which we abstractly know is finite or infinitely cyclic produces for us the data for the isomorphism, as defined in \cref{def:proofPi1Isomorphic}.

\begin{corollary}
    There exists an algorithm that
    \begin{itemize}
        \item takes as input a triangulated manifold $X$ whose fundamental group $\pi_1(X)$ is isomorphic to a finite group or $\Z$, and
        \item outputs data for such an isomorphism.
    \end{itemize} 
\end{corollary}

\begin{proof}
    First, we construct a presentation $\mathcal{P}$ of $\pi_1(X)$ in which every generator corresponds to a loop in the $1$-skeleton, given by a word of edges. Next, we enumerate a list of presentations $\{\mathcal{P}_i\}$, containing one presentation for $\Z$ and each finite group. Now, we apply Corollaries~\ref{cor:CompareGroups} and~\ref{cor:CompareGroups_infcyclic} diagonally to compare our presentation $\mathcal{P}$ to each $\mathcal{P}_i$, until we find an isomorphism from $\Z$ or a finite group to $\mathcal{P}$. 
\end{proof}

\subsection{Algorithms for combinatorial modified surgery}

We will show the existence of the following algorithms.

\begin{lemma}\label{alg:interForm}
There exists an algorithm that 
 \begin{itemize}
        \item takes as an input an oriented triangulation $\mathcal{K}$ of a closed, oriented, smooth $4$-manifold $X$, and a boolean variable that determines a ring $R$ that is either $\Z/2$ or $\Z$, and 
        \item outputs a matrix representing the $\Z/2$-intersection form of $X$
        \[\lambda^{\Z/2}_X\colon H_2(X;\Z/2)\times H_2(X;\Z/2) \to \Z/2\]
        if $R=\Z/2$ or outputs a matrix representing the $\Z$-intersection form of $X$
        \[\lambda^{\Z}_X\colon (H_2(X)/\Tors) \times (H_2(X)/\Tors) \to \Z\]
        if $R=\Z$.
\end{itemize}
\end{lemma}

\begin{proof}
In the following we discuss the case $R=\Z$. The case $R=\Z/2$ is dealt with in almost the same way.  First, from $\mathcal{K}$ we construct the simplicial chain and cochain groups along with boundary and coboundary operators.  Using this data, we then compute the homology and cohomology groups.  Note that cup and cap products are defined on, and can be computed on, the (co)chain level via standard formulae, and hence we can use these operations in our algorithm.  Using elementary linear algebra
    we determine a basis for $H_2(\mathcal{K};\Z)/\Tors$
    and for $H^2(\mathcal{K};\Z)/\Tors$.  The fundamental class of $X$ is given by the sum of all oriented 4-dimensional simplices of $\mathcal{K}$, and capping with it induces the Poincar\'{e} duality isomorphism $H^2(\mathcal{K};\Z)\to H_2(\mathcal{K};\Z)$
    which descends to an isomorphism $H^2(\mathcal{K};\Z)/\Tors\to H_2(\mathcal{K};\Z)/\Tors$.  The inverse to this isomorphism can be computed, e.g.\ by computing the isomorphism with respect to the above bases, and then computing the inverse matrix. Use this Poincar\'{e} duality inverse map and the standard cup product formula (which is known to descend to cohomology) to obtain a map 
    $$H_2(\mathcal{K};\Z)/\Tors\times H_2(\mathcal{K};\Z)/\Tors\to H^4(\mathcal{K};\Z).$$  
    By evaluating on the fundamental class we obtain the intersection form
        \[\lambda_X^{\Z} \colon H_2(X;\Z)/\Tors\times H_2(X;\Z)/\Tors\to \Z.\]
    Since we already chose a basis for $H_2(X;\Z)/\Tors$  we obtain a matrix representing the intersection form.
\end{proof}

\begin{lemma}\label{alg:signature}
    There exists an algorithm that 
    \begin{itemize}
        \item takes as input an oriented triangulation $\mathcal{K}$ of a closed, oriented, smooth $4$-manifold $X$, and
        \item outputs the signature $\sigma(X)$ of the $\Z$-intersection form $\lambda_X\colon H_2(X;\Z)\times H_2(X;\Z)\to \Z$.
    \end{itemize}
\end{lemma}

\begin{proof}
    We obtain a matrix representing the intersection form of $X$ via \cref{alg:interForm}. Then we can use \cref{cor:alg-for-signature} to compute the signature.
\end{proof}

\begin{remark}
    An alternative way to compute the signature is to utilise the combinatorial formula given by Ranicki--Sullivan~\cite{RanSulSignature}.  As in the proof of \cref{alg:interForm} we can construct simplicial chain and cochain groups along with boundary and coboundary operators
    \begin{align*}
        \partial &\colon C_{i}\to C_{i-1} \\
        \delta &\colon C_{i} \to C_{i+1}. 
    \end{align*}
    where we have identified the cochain groups with the chain groups using the choice of orientation of the simplices.  Ranicki--Sullivan then define another combinatorial operator $\varphi\colon C_i\to C_{4-i}$ such that we can form a matrix operator
    \[
    \begin{pmatrix}
        \varphi & \delta \\
        \partial & 0
    \end{pmatrix}\colon C_{2}\oplus C_{3} \to C_{2}\oplus C_{3}
    \]
    and they prove that the signature of this matrix gives the signature of the manifold.  Compute the signature of this matrix via a diagonalisation algorithm. Note that this can be done entirely working over the integers, as explained in \cref{cor:alg-for-signature}. 
\end{remark}    

\begin{lemma}\label{alg:normalData}
    Let $\pi$ be a finite group with its multiplication table $T_{\pi}$.   
    There exists an algorithm that
    \begin{itemize}
        \item takes as an input an oriented triangulation $\mathcal{K}$ of a  closed, oriented, smooth $4$-manifold $X$ along with data for some isomorphism $\pi_1(X)\cong \pi$ $($see \cref{def:proofPi1Isomorphic}$)$, and
        \item outputs the 1-type data for $X$ along with the reduced 1-smoothing data.
    \end{itemize}
\end{lemma}

\begin{proof}
    The first entry needed for the 1-type data ($T_{\pi}$) is already given to us.  To build the truncated classifying space, we apply the construction used in \cref{lem:simplicialMap} but halt the bar construction after adding the $5$-simplices. We now determine the final element of the 1-type data, $w_2^{\pi}$, which requires the most effort.

    The first step is to build a triangulation $\widetilde{\mathcal{K}}$ which models the universal cover of $M$.  Start by choosing a maximal tree $T$ in $\mathcal{K}$ and taking $n:=\lvert\pi\rvert$ disjoint copies of it and labelling these by group elements $g_1,\dots, g_n$.  Denote the result by $\widetilde{\mathcal{K}}_0$.  Every remaining $1$-simplex $e$ in $\mathcal{K}$ corresponds to a group element $g(e)$.  Suppose $e$ is a $1$-simplex between vertices $v_1$ and $v_2$ and let $v^{g_k}_1$ and $v^{g_k}_2$ denote the various lifts.  Then form $\widetilde{\mathcal{K}}_1$ by, for every 1-simplex $e=(v_1,v_2)$ not in $T$ we glue in $n$ 1-simplices $e^{g_k}=(v_1^{g_k},v_2^{g(e)\cdot g_k})$.  We glue the $k$-simplices in $\mathcal{K}$ for $2\leq k\leq 4$ dimension by dimension. The attaching maps of these simplices lift uniquely, once any lift of any vertex of its boundary is fixed. There are $n$ such lifts. Thus we have constructed $\wt{\mathcal{K}}$.

    Next, we recall that the Wu formula says that in an oriented $4$-manifold $Y$, for every $y\in H^2(Y;\Z/2)$, we have $\langle w_2(Y),y\rangle=\lambda^{\Z/2}(y,y)$, see for example~\cite[Exercise~5.7.3]{GS}. From the universal coefficient theorem with $\Z/2$ coefficients it follows that this formula determines the Stiefel--Whitney class $w_2(Y)$. Thus we can compute $w_2({\mathcal{K}})$ and $w_2(\widetilde{\mathcal{K}})$ by using \cref{alg:interForm} to compute the parity of the $\Z_2$-intersection form of $\mathcal{K}$ and $\widetilde{\mathcal{K}}$.
        
If $w_2(\widetilde{\mathcal{K}})\neq 0$ then the 1-type data for $M$ has $w_2^\pi =\infty$.  If $w_2(\widetilde{\mathcal{K}})=0$ then the algorithm proceeds as follows.
Realise our given isomorphism $\pi_1(X)\to \pi$ by a map $X\to B\pi^{(5)}$ (\cref{lem:simplicialMap}). From the Serre spectral sequence, one can deduce that the following sequence is exact
\[\begin{tikzcd}
    H^2(B\pi;\Z/2)\ar[r]& H^2(X;\Z/2)\ar[r]& H^2(\wt{X};\Z/2)\ar[r]&0
\end{tikzcd}\]
define $w_2^{\pi}\in H^2(B\pi^{(5)};\Z/2)$ to be the unique pre-image of $w_2(X)$.

Using the multiplication table $T_{\pi}$ one can find the finite set of automorphisms $\Aut(\pi)$. Pre-composing these with the given map $\pi_1(X)\to\pi$ gives us the finite set $\Isom(\pi_1(X),\pi)$ of all isomorphisms $f\colon\pi_1(X)\to \pi$. To generate all reduced 1-smoothing data, apply \cref{lem:simplicialMap} for each isomorphism $f$ in $\Isom(\pi_1(X),\pi)$, to obtain a simplicial map  $c_f\colon B\pi_1(X)^{(5)}\to B\pi^{(5)}$ with $(c_f)_* = f$.  Output only those that give $\iota^*(Bf)^{*}(w_2^{\pi})=w_2(X)$.  
\end{proof}

Next, we give a computational method to decide the vanishing of the primary invariant $\pri\colon \Omega_4^{\xi}\to H_4(B\pi)$.

\begin{lemma}\label{alg:primaryInvariant}
    Let $\pi$ be a finite group with multiplication table $T_{\pi}$. There exists an algorithm to calculate the primary invariant $\pri\colon \Omega_4^{\xi}\to H_4(B\pi)$, that
\begin{itemize}
    \item takes as input two oriented triangulations $\mathcal{K}$ and $\mathcal{L}$ of two closed, oriented, smooth $4$-manifolds $M$ and $N$ respectively with the same 1-type data, along with both of their reduced smoothing data and isomorphisms $\pi_1(M)\cong \pi$ and $\pi_1(N)\cong \pi$ $($\cref{def:proofPi1Isomorphic}$)$, and
    \item outputs whether or not there exists reduced 1-smoothings $c_M$ and $c_N$ such that \[ (c_N)_*^{-1}\left((c_M)_*(\iota_M)_*[M]\right) = (\iota_N)_*[N],\] where $\iota_M\colon M\to B\pi_1(M)$ and $\iota_N\colon N\to B\pi_1(N)$ are the Postnikov truncations. 
\end{itemize}
\end{lemma}

For the proof of \cref{alg:primaryInvariant} we will need the following simple algebraic fact. 

\begin{lemma}\label{lem:integralLinearEquations}
   There exists an algorithm that
    \begin{itemize}
        \item takes as input an $(m\times n)$-matrix $A$ with integer coefficients and a vector $b\in \Z^{n}$, and
        \item decides whether or not there exists a solution $x\in \Z^m$ of 
        $Ax=b$, and if such a solution exists outputs a single solution.
    \end{itemize}
\end{lemma}

\begin{proof}
    Use Smith normal form to obtain a decomposition $A=PDQ$ for $P,Q$ some invertible square integer matrices and $D$ a diagonal rectangular matrix, i.e.\ $D_{ij}=0$ whenever $i \neq j$.  Apply the algorithm for the inverse of the matrix $P$ to obtain an equation we can easily solve 
   \[D(Qx)=P^{-1}b.\qedhere\]
\end{proof}

\begin{proof}[Proof of \cref{alg:primaryInvariant}]
Use the given group isomorphism data, along with the multiplication table $T_{\pi}$ to generate the finite set $\Isom(\pi_1(\mathcal{L}),\pi_1(\mathcal{K}))$ of all possible group isomorphisms. 

    For each $f\in \Isom(\pi_1(\mathcal{L}),\pi_1(\mathcal{K}))$ we construct a simplicial map $Bf \colon B\pi_1(\mathcal{L})^{(5)}\to B\pi_1(\mathcal{K})^{(5)}$ (see \cref{lem:simplicialMap}). This gives us a map $(Bf)\circ \iota_\mathcal{L} \colon \mathcal{L}\to B\pi_1(\mathcal{K})^{(5)}$. Similarly, we make a simplicial map $\iota_{\mathcal{K}} \colon \mathcal{K}\to B\pi_1(\mathcal{K})^{(5)}$ which induces the identity on the fundamental groups.

    For a finite triangulated $4$-dimensional  manifold, its fundamental class is the class given by the sum of all of its top dimensional oriented simplices. Set 
    $$\alpha :=(\iota_{\mathcal{K}})_*[\mathcal{K}] - ((Bf) \circ (\iota_{\mathcal{L}})_*)[\mathcal{L}]\in C^{\simp}_{4}(B\pi_1(\mathcal{K})^{(5)}),$$
    which is a cycle in the simplicial chain complex.

    The simplicial complex $C^{\simp}_*(B\pi_1(\mathcal{K})^{(5)})$ is finite in each degree, hence the differential
    \[C_5^{\simp}(B\pi_1(\mathcal{K})^{(5)})\to C_4^{\simp}(B\pi_1(\mathcal{K})^{(5)})\]
    is given by a finite dimensional matrix $A$. We apply \cref{lem:integralLinearEquations} to solve the following problem. 
    \[\text{Does } Ax=\alpha \text{ for some } x\in C_5^{\simp}(B\pi_1(\mathcal{K})^{(5)})?\]
    If such an $x\in C_5^{\simp}(B\pi_1(\mathcal{K})^{(5)})$ exists we return `yes', otherwise we move on to the next isomorphism. If all isomorphisms have been exhausted, we return `no'.
\end{proof}

\subsection{Proof of \texorpdfstring{\Cref{thm:mainAlmostSPinOrSpin,thm:mainNon-Spin}}{}}

In this section, we prove the main theorems on stable classification. We will be relying on the results of Teichner's thesis~\cite{teichnerthesis}.

\begin{proof}[Proof of \cref{thm:mainNon-Spin}]
    In the case of \eqref{item:thm-8-4-i}, start by generating a certificate (\cref{def:proofPi1Isomorphic}) that the fundamental groups are both infinite cyclic by applying \cref{cor:CompareGroups_infcyclic}.  Similarly, in the case of \eqref{item:thm-8-4-ii}, start by generating a certificate that the fundamental groups of $X_1$ and $X_2$ are isomorphic to some group $\pi$ by applying \cref{cor:CompareGroups}. If the fundamental groups are not isomorphic then our manifolds are not stably diffeomorphic.
    
    By a calculation using the Atiyah--Hirzebruch spectral sequence (see e.g.\ \cite[p.~10]{teichnerthesis}) we have $\xi\colon\BSO\times B\pi\to \BSO$ and $$\begin{tikzcd}
        \Omega_4^{\xi}\ar[r,"\cong"']\ar[r,"\pri"]&\Omega_4^{SO}\oplus H_4(B\pi)\ar[r,"\cong"']\ar[r,"\sigma\oplus \id"]&\Z\oplus H_4(B\pi).
    \end{tikzcd} $$
Either the geometrical dimension of $\pi$ is at most $3$, implying that $H_4(B\pi;\Z)=0$ and two such totally non-spin manifolds $M,N$ are orientation-preserving stably diffeomorphic if and only if $\sigma(M)=\sigma(N)$, which is computable (\cref{alg:signature}).
    Or, $\pi$ is finite and we apply Algorithms \ref{alg:signature} and \ref{alg:primaryInvariant}.
\end{proof}

\begin{proof}[Proof of \cref{thm:mainAlmostSPinOrSpin}]

In the case of \eqref{item:thm-8-3-ii} start by generating a certificate (\cref{def:proofPi1Isomorphic}) that the fundamental groups of $X_1$ and $X_2$ are isomorphic to some group $\pi$ by applying \cref{cor:CompareGroups}. If the fundamental groups are not isomorphic then our manifolds are not stably diffeomorphic.
    
    We apply \cref{alg:normalData} to compute the 1-type data and reduced smoothing data for both of our given manifolds.  If they do not have the same 1-type data then they are not stably diffeomorphic, so we assume that they do from now on.
    
    It is sufficient to show that in all of these cases the map
\begin{equation*}
\Omega_4^{\xi}\xrightarrow{\sigma+\pri} \Z\oplus H_4(B\pi)  
\end{equation*}
is an injection.
Then we finish the proof by applying the algorithms in  \cref{alg:primaryInvariant,alg:signature} or, in the case of $\pi$ infinite cyclic, only Algorithm \ref{alg:signature}.

The case \eqref{item:thm-8-3-i} of $\pi$ infinite cyclic is handled by using the Atiyah--Hirzebruch spectral sequence for types (I) and (II). Since $H^2(B\Z;\Z/2)=0$, type (III) is impossible.

For the case of $\pi$ a finite group with cyclic Sylow $2$-subgroup, we use the result of \cite[Theorem 4.4.4.]{teichnerthesis} stating that the following maps given by the signature and the invariant $\pri$ are isomorphisms
\[\begin{array}{rcll}
    \Omega_4^{\xi}&\cong&\Z\oplus H_4(B\pi) & \text{for type (I)},\\
    \Omega_4^{\xi}&\cong& 16\Z\oplus H_4(B\pi) & \text{for  type (II)},\\
    \Omega_4^{\xi}&\cong& 8\Z\oplus H_4(B\pi) & \text{for  type (III) and $\pi$ of even order},\\
    \Omega_4^{\xi}&\cong& 16\Z\oplus H_4(B\pi) & \text{for type (III) and $\pi$ of odd order}.  
\end{array}\]
\end{proof}


\let\MRhref\undefined
\bibliographystyle{hamsalpha}
\bibliography{bib.bib}

\vspace{-7cm}

\end{document}